\documentclass{article}

\usepackage{arxiv}

\usepackage[utf8]{inputenc} 
\usepackage[T1]{fontenc}    
\usepackage{hyperref}       
\usepackage{url}            
\usepackage{booktabs}       
\usepackage{amsfonts}       
\usepackage{nicefrac}       
\usepackage{microtype}      
\usepackage{lipsum}
\usepackage{graphicx}
\graphicspath{ {./images/} }

\usepackage{amsmath,amssymb,amsthm,mathtools}
\usepackage{mathrsfs} 
\usepackage{ascmac}
\usepackage{graphicx,color}

\newtheorem{Th}{Theorem}
\newtheorem{Le}{Lemma}
\newtheorem{Co}{Corollary}

\newcommand{\N}{\mathbb{N} }

\newcommand{\R}{\mathbb{R} }

\newcommand{\LA}{\left \langle }
\newcommand{\RA}{\right  \rangle }

\newcommand{\LC}{\left ( }
\newcommand{\RC}{\right ) }
\newcommand{\LD}{\left \{ }
\newcommand{\RD}{\right \} }

\newcommand{\DS}{\displaystyle }
\newcommand{\LN}{\left \|  }
\newcommand{\RN}{\right \| }
\DeclareMathOperator*{\esssup}{ess\sup}
\DeclareMathOperator{\sgn}{sgn}

\title{Doubly nonlinear parabolic equation with perturbation}

\author{
 Shun Uchida \\
    Department of Science and Technology, \\
    Faculty of Science and Technology, \\ 
    Oita University,\\
    700 Dannoharu, Oita City, Oita Pref., 
    870-1192, JAPAN.\\
  \texttt{shunuchida@oita-u.ac.jp} \\
}

\begin{document}
\maketitle
\begin{abstract}
In this paper, we consider the initial boundary value problem of a doubly nonlinear parabolic equation with nonlinear perturbation. 
We impose the homogeneous Dirichlet condition on this problem. 
We aim to reduce the growth condition of the nonlinear term and the largeness of exponent as possible 
so that we can apply our result to both singular and degenerate type parabolic equations.  
We use in our proof the $L ^{\infty }$-estimate of the time-discrete equation derived in the previous work 
and apply the so-called $L ^{\infty }$-energy method to the time-discrete problem.
We also discuss the uniqueness of solution by using the previous result. 
\end{abstract}

\keywords{Doubly nonlinear parabolic equation \and initial boundary value problem \and homogeneous Dirichlet boundary condition \and $p$-Laplacian \and $L ^{\infty }$-energy method \and subdifferential}


\section{Introduction}

In this paper, we are concerned with the initial boundary value
problem of the following doubly nonlinear equation with a perturbation term:
\begin{equation}
\begin{cases}
\partial _t \beta ( u (x, t )) - \nabla \cdot \alpha ( x , \nabla   u (x ,t ) ) 
		= F (\beta  ( u (x,t ) )) + f(x,t ) 
	~~&~~  (x, t )\in Q := \Omega \times (0,T) ,  \\
u (x, t ) = 0  ~~&~~  (x, t )\in \partial  \Omega \times (0,T) , \\
u (x, 0  ) = u_ 0 (x  ) ~~&~~ x \in  \Omega , 
\end{cases}
\label{P}
\tag{P}
\end{equation}
where $\Omega \subset \R ^d $ ($d\geq 1 $) be
a bounded domain with a sufficiently smooth boundary $\partial \Omega $. 
Throughout this paper, we impose the followings on $\alpha $  and $\beta $:
\begin{itemize}
\item[(H.$\alpha $)]  
There exists a $C ^1 $-class function $a : \Omega \times \R ^ d  \to \R $ such that 
$ a (x, \cdot ) : \R ^d \to \R $ is convex and $\alpha (x,  \cdot  ) := D _z a (x, \cdot  )$ 
holds for a.e. $x\in \Omega $.
In addition, 
$a$ and its derivative $\alpha : \Omega \times \R ^d \to \R ^d $ 
satisfy the followings
with some exponent 
$p \in (1, \infty )$ and some constants $c,C > 0$:
\begin{align}
&c |z | ^ p - C  \leq a (x , z ) \leq C  ( |z | ^p  +1  ) , \label{A01} \\
&|\alpha  (x,z ) | \leq C (|z | ^{p-1 }  +1 ) , ~~~\alpha (x, 0 ) = 0 , \label{A02}
\end{align}
and 
\begin{equation}
\begin{split}
(\alpha (x , z _1 ) - \alpha (x,  z_2 )) \cdot (z _1 - z_2 ) \geq c |z _1 - z _ 2 | ^p ~~&~~\text{if } p \geq 2 , \\
(\alpha (x , z _1 ) - \alpha (x,  z_2 )) \cdot (z _1 - z_2 ) \geq \frac{ c |z _1 - z _ 2 | ^2}{ |z _1 | ^{2-p } + |z _2 | ^{2-p }  }  ~~&~~\text{if } 1 < p < 2 ,
\end{split}
\label{A03}
\end{equation}
for every $z , z_ 1 , z_ 2 \in \R ^d $ and almost every $x \in \Omega $.

\item[(H.$\beta $)] 
There exists a $C ^1 $-class convex function $j  :  \R  \to ( - \infty , + \infty ]  $ such that 
$\beta  :=  j '   $ 
holds and $\beta : \R \to \R $ satisfies $\beta (0 ) = 0 $.
\end{itemize}

A typical example of the main term $\nabla \cdot \alpha ( \cdot , \nabla u)$ 
is the so-called $p$-Laplacian 
$\Delta _p u := \nabla \cdot (| \nabla u | ^{p-2 }  \nabla u ) $,
which satisfies (H.$\alpha $) with $a (x ,z ) = \frac{1}{p} |z| ^p  $ 
and $\alpha (x ,z ) = |z| ^{p-2} z $ ($p\in (1, \infty )$). 
The sum of a finite number of $p_i$-Laplacian ($i = 1, \ldots ,n$) also fulfills (H.$\alpha $) with $p = \max _ {i =1 , \ldots, n }p_i$. 
By (H.$\beta $), we can see that $\beta $ is non-decreasing function (see the next section). 
Remark that there might be some $s \in \R $ such that $\beta (s), \beta ^{-1} (s) = \varnothing $, 
namely, we can deal with the case where, for instance, 
\begin{align*}
\beta (s) & = \tan s ~~~~~~~& s \in (- \pi /2 , \pi /2 ) ,  \\   
\beta (s) & =  -s / (s+1 ) ~~~~~~~& s \in (-1  , \infty  ) ,  \\
\beta (s) & =  \log (s +1) ~~~~~~~& s \in (- 1  , \infty  ) .  
\end{align*}

As for the case where $F \equiv 0$ (without perturbation terms),
there are numerous previous works. 
For instance, \cite{Ba, BarV, GM} discuss the existence of solution to the doubly nonlinear equation
and its uniqueness is considered in \cite{Car, U}. 
On the other hand,  there are a few results for the case where $F \not \equiv 0$. 
For example, 
\cite{AL} show the existence of solution to \eqref{P} with general nonlinearity $\beta $ 
under some growth condition of $F$. 
In \cite{G}, the case where $ p \geq 2 $ and $F$ is bounded by some polynomial is considered. 
In \cite{HO, HM}, they deal with the solvability of \eqref{P} and existence of global attractor 
under the locally Lipschitz continuity of $\beta $ and some growth condition of $F$.

The equation \eqref{P} is classified as a degenerate/singular type parabolic equation 
depending on the conditions on $\beta $ and the largeness of exponent $p$.
Especially, 
we can see that the solution to the singular parabolic equations might become zero at a finite time, 
which is the so-called extinction phenomena. 
On the other hand, it is well known that the growth term $F (\beta (u))$ might cause 
the blow-up of the solution in finite time. 
Hence in the equation \eqref{P}, 
the conflict of the extinction and the blow-up might occurs 
and then the time global behavior of solutions seems to be an interesting problem. 
However, in many papers, 
they assume some condition to $\beta , p , F $
and they restrict themselves to the case where \eqref{P} is a degenerate type parabolic equation.

 In this paper, we aim to reduce the assumption of $\beta $ 
and the growth condition $F $, and the largeness of $p $ as possible 
and assure the solvability of \eqref{P}. 
 We use in our proof the $L ^{\infty }$-estimate of the time-discrete equation 
 derived in \cite{U}
and apply the so-called $L^{\infty }$-energy method to the  time-discrete problem. 
More precisely, we shall prove the following results in this paper 
(notation is defined in the next section). 
\begin{Th}
\label{Th01} 
Assume (H.$\alpha $) and (H.$\beta $).
Let $F : \R \to \R $ is continuous and $\beta ^{-1} : \R \to \R $
is locally Lipschitz continuous. 
Then for every $u _ 0 \in L ^{\infty } (\Omega  )$ with $ \beta (u _ 0) \in L ^{\infty } (\Omega  )$
and $f \in L ^{\infty } (Q  )$, 
there exists $ T' \in ( 0 , T ] $ such that 
\eqref{P} has at least one solution which satisfies 
\begin{align*}
u &\in L^{\infty } (0,T' ; W^{1, p } _ 0  (\Omega )) \cap 
		L^{\infty } (0,T' ; L ^{\infty }  (\Omega ) )  \cap  W^{1, 2 }   (0,T'  ; L^2   (\Omega )) ,\\
\beta (u) &\in 
		L^{\infty } (0,T' ; L ^{\infty }  (\Omega ) )  \cap W^{1, p'  }   (0,T'  ; W ^{-1, p' } (\Omega )),
\end{align*}
and 
fulfill \eqref{P}  in $L ^{p'} ( 0,T' ; W ^{-1 , p ' } (\Omega ) ) $. 
\end{Th}

\begin{Th}
\label{Th02}
Assume (H.$\alpha $) and (H.$\beta $).
Let $F : \R \to \R $ is continuous and $\beta : \R \to \R $
is locally Lipschitz continuous. 
Then for every $u _ 0 \in L ^{\infty } (\Omega  )$ with $ \beta (u _ 0) \in L ^{\infty } (\Omega  )$
and $f \in L ^{\infty } (Q)$, 
there exists $ T' \in ( 0 , T ] $ such that 
\eqref{P} has at least one solution  which satisfies 
\begin{align*}
u &\in L^{p } (0,T' ; W^{1, p } _ 0  (\Omega )) \cap 
		L^{\infty } (0,T' ; L ^{\infty }  (\Omega ) ) , \\
\beta (u) &\in 
		L^{\infty } (0,T' ; L ^{\infty }  (\Omega ) )  \cap W^{1, p'  }   (0,T'  ; W ^{-1, p' } (\Omega )),
\end{align*}
and 
fulfill \eqref{P}  in $L ^{p'} ( 0,T' ; W ^{-1 , p ' } (\Omega ) ) $. 
\end{Th}

By appending additional condition to the assumptions in Theorem \ref{Th02},
we can slightly recover the regularity of solution. 
\begin{Th}
\label{Th03} 
In addition to the assumptions in Theorem \ref{Th02}, 
assume that $D (\beta ) = \R $ and $F $ is monotone increasing, and $f \equiv 0 $. 
Then for every $u _ 0 \in L ^{\infty } (\Omega  )  \cap W ^{1, p } _ 0  (\Omega ) $ 
		with $ \beta (u _ 0) \in L ^{\infty } (\Omega  )$, 
there exists $ T' \in ( 0 , T ] $ such that 
\eqref{P} has at least one solution  which satisfies 
\begin{align*}
u \in   L^{ \infty  } (0,T' ; W^{1, p } _ 0  (\Omega ))  \cap 
		L^{\infty } (0,T' ; L ^{\infty }  (\Omega ) )  \\
\beta (u) \in 
		L^{\infty } (0,T' ; L ^{\infty }  (\Omega ) )  \cap W^{1, p'  }   (0,T'  ; W ^{-1, p' } (\Omega ))
\end{align*}
and 
fulfill \eqref{P}  in $L ^{p'} ( 0,T' ; W ^{-1 , p ' } (\Omega ) ) $. 
\end{Th}

Our result can be apply, for instance, to the following doubly nonlinear parabolic equation
with a perturbation:
\begin{equation}
\begin{cases}
\partial _t   | u (x, t ) | ^{q-2} u (x, t )  - \Delta _p  u (x ,t )  
		=F  (  u (x, t ) ) ~~&~~ (x, t )\in Q := \Omega \times (0,T) , \\
u (x, t ) = 0  ~~&~~ (x, t )\in \partial  \Omega \times (0,T) , \\
u (x, 0  ) = u_ 0 (x  ) ~~&~~ x \in  \Omega .  
\end{cases}
\label{doubly-nonlinear}
\end{equation}

In \cite{FO, Ishii, Tsutsumi}, they deal with the case where $p \geq 2 $, $q =2 $ and $F (s) = |s | ^{r-2}s $.
In \cite{YLJ}, they consider the case where $  ( p-1) / (q -1) \geq 1 $ and $F (s) = |s | ^{r-2}s $
and in  \cite{KMN, MZC} consider the case where $ 2 \leq  p < d  $ and $F (s) = |s | ^{r-2}s $.
In \cite{LL}, they discuss the blow-up of the solution for 
the case where $  ( p-1) / (q -1) > 1 $ and $F (s) = s  ^{r} \log s  $.
To the best of our knowledge,
 the solvability for the case where \eqref{doubly-nonlinear} is classified into the singular parabolic equation, 
i.e. $p < 2 $ and  $q >2$ has been still open.  
By our main theorem, we can solve \eqref{doubly-nonlinear}
for every exponent $ p , q > 1$ and general nonlinearity $F $.  
\begin{Co}
Let $ p, q > 1$ and $F : \R \to \R $ be a continuous function. 
Then for every $u _ 0 \in L ^{\infty } (\Omega  )$, 
there exists $ T' \in ( 0 , T ] $ such that 
\eqref{doubly-nonlinear} has at least one solution  which satisfies if $q < 2 $,  
\begin{align*}
u &\in L^{\infty } (0,T' ; W^{1, p } _ 0  (\Omega )) \cap 
		L^{\infty } (0,T' ; L ^{\infty }  (\Omega ) )  \cap  W^{1, 2 }   (0,T'  ; L^2   (\Omega )) ,\\
|u| ^{q-2 } u  &\in  W^{1, p'  }   (0,T'  ; W ^{-1, p' } (\Omega )),
\end{align*}
and if $q \geq 2 $, 
\begin{align*}
u &\in L^{p } (0,T' ; W^{1, p } _ 0  (\Omega )) \cap 
		L^{\infty } (0,T' ; L ^{\infty }  (\Omega ) ) , \\
|u| ^{q-2 } u  &\in   W^{1, p'  }   (0,T'  ; W ^{-1, p' } (\Omega )).
\end{align*}
\end{Co}

\section{Preliminary}

Henceforth, we shall use the following notations  
in order to denote the standard Lebesgue and Sobolev space, respectively:
\begin{align*}
& \DS L ^r (\Omega ) := \LD  v  : \Omega  \to \R   ; 
~
\begin{matrix}
v  \text{ is Lebesgue measurable and }  \\[2mm]
\DS 
		\| v \| _{L ^r (\Omega ) } := 
		\LC  \int_{\Omega }   | v (x) | ^r   dx  \RC ^{1/r } < \infty .
\end{matrix}
 \RD ~~~~~~r \in [1 , \infty ) , \\[2mm]
&\DS L^{\infty } (\Omega ) := \LD  v  
 : \Omega   \to \R  ; 
~
\begin{matrix}
 v  \text{ is Lebesgue measurable and  }  \\[2mm]
\DS  \| v \| _{L^{\infty } (\Omega  ) } :=  \esssup _{  x \in \Omega  }  | v (x) |   < \infty .
\end{matrix}
 \RD  , \\[2mm]
 &\DS W^{1, r } (  \Omega ) := \LD  v  \in L^r ( \Omega  );
			~~ \nabla  v \in L^r (\Omega  )  ~~\forall i =1 , \ldots, N  \RD ,~~~~~ r \in [1 , \infty ] ,
\end{align*}
where $\nabla v $ is the partial derivatives of $v$ in the distributional sense.
Let $ C ^{\infty } _{0} (\Omega ) $ stand for the set of infinitely and continuously differentiable functions $v : \Omega \to \R $
with compact support in $\Omega $. 
Then 
$W^{1, r } _ 0 (  \Omega ) $ is defined by the completion of $C ^{\infty } _ 0 (\Omega )$ 
in $W ^{1, r } (\Omega )$ with norm $ \| v  \| _{W ^{1 , r }} := \| v \| _{ L ^{r } } + \|  \nabla v \| _{L ^r }$. 
Moreover, 
$W^{- 1, r '  }  (  \Omega ) $ is the dual space of 
$W^{1, r } _ 0 (  \Omega ) $ and $ \LA \cdot , \cdot \RA  _{W ^{1,p }}$ denotes 
the duality pairing between 
$W^{1, r } _ 0 (  \Omega ) $  and $W^{- 1, r '  }  (  \Omega ) $,
where $r ' = r / (r -1 )$ is the H\"{o}lder conjugate exponent of $r \in (1, \infty )$. 
The norm of $W^{- 1, r '  }  (  \Omega ) $ is  defined as 
\begin{equation*}
\| v \| _{W^{- 1, r '  }  (  \Omega ) } := \sup _{
\phi \in W ^{1, p } _0   (\Omega ) , ~~ \| \phi \| _{W ^{1, p }   (\Omega )} =1
 }
 \LA v , \phi  \RA _{W ^{1,p }} . 
\end{equation*}

For a Lebesgue measurable function  $ v : \Omega \times (0, T) \to \R $, we define the Bochner space by 
\begin{align*}
& \DS L ^r (0 , T ; X  ) := \LD  v  : Q \to \R   ; 
~
\begin{matrix}
v  \text{ is Lebesgue measurable and }  \\[2mm]
\DS 
		\| v \| _{L ^r ( 0, T; X  ) } := 
		\LC  \int_{0}^{T}   \| v (t ) \| ^r  _X   dt  \RC ^{1/r } < \infty .
\end{matrix}
 \RD  ~r \in [1 , \infty )  \\
&
 \DS L ^\infty  (0 , T ; X  ) := \LD  v  : Q \to \R   ; 
~
\begin{matrix}
v  \text{ is Lebesgue measurable and }  \\[2mm]
\DS 
		\| v \| _{L ^\infty  ( 0, T; X  ) } := 
		  \esssup _{t \in (0,T )}  \| v (t ) \|   _X  < \infty .
\end{matrix}
 \RD , 
\end{align*}
where $X $ is the Lebesgue or Sobolev space given above and $ Q := \Omega \times (0,T )$. 
When $X = L ^r (\Omega )$, we shall abbreviate $L ^r (0,T ; L ^r (\Omega )) $ to $ L ^r (Q )$. 
Moreover, 
we define $ W ^{1, q } (0,T ; X )$ by 
\begin{equation*}
 \DS W ^{1, q } (0 , T ; X  ) := \LD  v  : Q \to \R   ; 
~
\begin{matrix}
v  \text{ is time differentiable in the distributional sense }  \\[2mm]
\DS 
	\text{ and }  	\| v \| _{ W ^{1, q } ( 0, T; X  ) }   < \infty .
\end{matrix}
 \RD , 
\end{equation*}
where $\| v \| _{ W ^{1, q } ( 0, T; X  ) } := 
		\| v  \|   _ { L ^q ( 0 ,T ; X ) } + \| \partial _ t v  \|   _ { L ^q ( 0 ,T ; X ) }    $
		and $\partial _t v $ is the time derivative of $v$ in the distributional sense.

For later use, we here state some basic properties of convex functions and its subdifferential
(see e.g., \cite{Bar,PUni,Bre}). 
In general setting, let $X $ be a Banach space and $\phi : X \to ( - \infty , + \infty ] $ be a real valued function 
which possibly takes $\phi  (u ) = + \infty $ with some $u \in  X $. 
We call the set $D(\phi ) := \{ u \in X  ; \phi  (u ) < + \infty \}$ the effective domain of $\phi $. 
When $\phi  $ is a convex and lower semi-continuous function with $D(\phi ) \neq \varnothing $,
we can define the subgradient of $\phi  $ at $u _ 0   \in D(\phi  )$ by 
\begin{equation*}
\partial \phi  (u _ 0 ) := \{ \eta  \in X ^{\ast} ; ~ \LA \eta , v - u _ 0  \RA _X 
 \leq  \phi (v  ) - \phi  (u _ 0 ) ~~\forall  v \in X \} , 
\end{equation*}
where $X ^{\ast}$ is the dual space of $X$ and $\LA \cdot , \cdot \RA _{X}$
is the duality pairing between $X$ and  $X ^{\ast}$. 
The subdifferential operator $\partial \phi $ from $X $ into $2 ^{X ^{\ast} }$ (the power set of $X ^{\ast} $)
is the mapping 
$ \partial \phi  : u  \mapsto \partial \phi (u  )$. 
Since there might be multiple elements $\eta  $ satisfying $\eta  \in \partial \phi (u_ 0 ) $ in general,  
then the subdifferential operator $ \partial \phi  $ possibly becomes a set-valued mapping.
By the definition of the subgradient, it is obvious that 
$0 \in \partial \phi  (u )$ if and only if $\phi  $ attains its  minimum at  $u  $.

It is well known that the subdifferential operator becomes maximal monotone. 
Here the set-valued operator $A  : X  \to 2 ^{X ^{\ast }}$ is said to be monotone 
if 
\begin{equation}
\LA \eta _1 - \eta _2 , u _1 - u_2 \RA _{X} \geq 0 ~~~\forall u_1 ,u_2 \in D(A ),
\forall \eta _1 \in A ( u _ 1 ) , \forall \eta  _2 \in A (u_ 2 )  
\label{Pre000}
\end{equation}
holds ($D(A ) \subset X $ is the domain of $A $)
and 
said to be maximal monotone 
if $R ( I + A ) = \R $ ($I $ is the duality mapping from $X $ to $X ^{\ast} $ and $R $ is the range of operator). 
Moreover,  maximal monotonicity of $ A $ is equivalent to the fact that
there is no monotonic expansion operator which contains $ A $ properly.

Let $j : \R \to ( - \infty , + \infty ] $ be a real valued function appears in (H.$\beta $). 
When $j $ is differentiable at $s _ 0 $ in usual sense, 
we can see that the usual derivative $ j'  (s _ 0 )$ coincides with the subgradient $\partial j  (s_ 0 )$. 
Hence by the assumption (H.$\beta $), $\beta (s)  = \partial j (s)$ holds in this paper
with $X = X ^{\ast} = \R $.
Since $\beta (0 ) = 0 $, $ j $ attains its minimum at $s = 0 $ and 
we can assume that $j (0 ) = 0 $ and $j \geq 0 $ without loss of generality. 
Here $\beta $ satisfies \eqref{Pre000}
with $X = X ^{\ast} = \R $, which implies that $\beta $ is a non-decreasing function.

Define $J : L ^p (Q) \to ( - \infty , + \infty ] $ by 
\begin{equation*}
J (u ) := \begin{cases}
~~ \DS \int_{0}^{T} \int_{\Omega } j (u (x,t )) dx dt ~~&~~\text{if }  u \in L ^p (Q) \text{ and } j (u) \in L ^1 (Q) , \\
~~ + \infty ~~& ~~ \text{otherwise.}
\end{cases}
\end{equation*}
Then this functional is convex and lower semi-continuous on $ L^q  (Q )$ 
and its subdifferential satisfies $\partial J (u ( x,t ) ) = \partial j  ( u ( x ,t )  ) = \beta (u (x,t ))$
for a.e. $(x,t )\in Q$, so $\beta $ can be regarded as a maximal monotone operator
from $L ^p (Q) $ onto $L ^{p'} (Q)$, where $p' = p /(p-1 )$ is the H\"{o}lder conjugate exponent of $p \in (1, \infty )$
(see, e.g., Proposition 3 of \cite{PUni}). Then we can see that the following holds
(see, e.g., Proposition 2.1 of \cite{Bar}): 
\begin{Le}
Let
 $\{ u_ n \} _{ n \in \N }$ weakly converge to $u$ in $L ^p (Q)$ 
and 
 $\{ \beta ( u_ n)  \} _{ n \in \N }$ strongly converge to $\xi $ in $L ^{p' } (Q)$. 
Then $\xi = \beta (u )$ for a.e. $\Omega \times (0,T )$.  
\label{Demi-beta}
\end{Le}

We next define the Legendre-Fenchel transform of $j $ by 
\begin{equation*}
j ^\ast ( \sigma ) 
= \sup _{s \in \R } \LC  \sigma s - j (s)\RC . 
\end{equation*}
It is well known that if $j $
is a convex lower semi-continuous function with $D (j )\neq \varnothing $,
then $j ^{\ast }$ 
is also convex lower semi-continuous and  satisfy $D (j ^{\ast} )\neq \varnothing $.
 Moreover, the subdifferential of $j ^{\ast }$ coincides with the inverse operator of $\beta $:
 \begin{equation*}
\partial j ^{\ast }  (\sigma  ) = \beta ^{-1} (\sigma ) 
:= 
\{ s \in \R ; ~ \sigma = \beta (s ) \} . 
\end{equation*}
Then by  $s \in \beta ^{-1}( \beta (s )) = \partial j ^{\ast} (\beta (s ))$, 
the definition of subgradient yields 
\begin{equation}
 ( \beta ( \sigma ) - \beta ( s )) s \leq j ^{\ast } ( \beta (\sigma )) - j ^{\ast }  ( \beta (s) ) . 
\label{Pre001} 
\end{equation}
Since $\beta (0 ) = 0 $, i.e., $0 \in \beta ^{-1} (0 )$,  $ j^{\ast} $ attains its minimum at $s = 0 $ and 
we can assume that $j ^{\ast} (0 ) = 0 $ and $j ^{\ast }\geq 0 $ without loss of generality.

We check the following lemma, which is a variant of Lemma 1.5 in Alt-Luckhaus \cite{AL}.
\begin{Le}
Assume $u \in L^p (0,T ; W ^{1,p } (\Omega ))$, 
$ \xi \in L ^{\infty } (Q) \cap  W  ^{1, p'} (0,T ; W ^{-1, p } (\Omega ))$, and 
$j ^{\ast } (\xi ) \in L^1 (Q)$.
If $\xi (x, t ) = \beta (u (x ,t )) $ for a.e. $ (x ,t ) \in Q$, 
then
\begin{equation}
\int_{\Omega } j ^{\ast} ( \xi ( t) )dx - \int_{\Omega } j ^{\ast} ( \beta (u _0 ) ) dx
= \int_{0}^{t} \LA \partial _t \xi (s) , u (s) \RA _{W ^{1,p } } dt  
\label{Lem-AL}  
\end{equation}
holds  for almost any $t \in (0 , T ]  $.
\label{IBP} 
\end{Le} 

\begin{proof}
We refer to the proof of Lemma 1.5 in \cite{AL}. 
Let $ u (\cdot  , t ) = u_ 0 $ for any $t < 0 $. 
Then by \eqref{Pre001}, for arbitrary $h > 0 $
\begin{equation}
\begin{split}
&j ^{\ast } (\beta  (u (x, s ))) - j ^{\ast } (\beta  (u (x, s -h )))  \\
&\hspace{3cm} \leq 
		 ( \beta  (u (x, s )) - \beta  (u (x, s -h  ))) u (x ,s ) , \\
&j ^{\ast } (\beta  (u (x, s ))) - j ^{\ast } (\beta  (u (x, s -h )))  \\
&\hspace{3cm} 
		\geq ( \beta  (u (x, s )) - \beta  (u (x, s -h  ))) u (x ,s - h  ) .
\end{split}
\label{Memo007-1} 
\end{equation}
Integrating the first inequality of \eqref{Memo007-1} over $\Omega \times (0,t )$, 
we have 
\begin{align*}
&\int_{0}^{t} \int _{\Omega }  j ^{\ast } (\beta  (u (x, s ))) dxds 
	-\int_{0}^{t} \int _{\Omega }    j ^{\ast } (\beta  (u (x, s -h ))) dx ds \\
&	=
\int_{t -h }^{t} \int _{\Omega }  j ^{\ast } (\beta  (u (x, s ))) dxds 
	-h   \int _{\Omega }  j ^{\ast } (\beta  (u_ 0  (x ))) dx , 
\end{align*}
namely, 
\begin{align*}
&\frac{1}{h} \int_{t -h }^{t} \int _{\Omega }  j ^{\ast } (\beta  (u (x, s ))) dxds 
	- \int _{\Omega }   j ^{\ast } (\beta  (u_ 0  (x ))) dx \\
&\leq \frac{1}{h} 
\int_{0}^{t} \int _{\Omega }  
	( \beta  (u (x, s )) - \beta  (u (x, s -h  ))) u (x ,s   )  dxds .
\end{align*}
Then we can replace $\beta  (u )$ with $\xi $ and obtain 
\begin{align*}
&\frac{1}{h} \int_{t -h }^{t} \int _{\Omega }  j ^{\ast } (\xi (x, s ))) dxds 
	-  \int _{\Omega }  j ^{\ast } (\beta  (u_ 0  (x ))) dx \\
& \leq \frac{1}{h} 
\int_{0}^{t} \int _{\Omega }  
	( \xi  (x, s ) - \xi (x, s -h  )) u (x ,s   )  dxds 
=
- \int_{0}^{t} 
	\LA  \frac{  \xi ( s -h  ) - \xi  ( s )  }{h }  , u ( s   ) \RA  _{W ^{1,p }}  ds 
\end{align*}
Since $\xi \in W ^{1 ,p ' } (0,T ; W ^{-1 , p'} (\Omega ))$, $\xi $ is differentiable for a.e. $t \in (0,T) $ and then 
\begin{align*}
&\limsup _{ h\to 0} \frac{1}{h} \int_{t -h }^{t} \int _{\Omega }  j ^{\ast } (\xi (x, s ))) dxds 
	- \int _{\Omega }   j ^{\ast } (\beta  (u_ 0  (x ))) dx \\
&\leq 
-\liminf _{h\to 0} \int_{0}^{t} 
	\LA  \frac{  \xi ( s -h  ) - \xi  ( s )  }{h }  , u ( s   ) \RA  _{W ^{1,p }}  ds  \\
&\leq 
-\int_{0}^{t} 
	\LA - \partial _t  \xi ( s   )  , u ( s   ) \RA  _{W ^{1,p }}  ds  
=\int_{0}^{t} 
	\LA  \partial _t  \xi ( s   )  , u ( s   ) \RA  _{W ^{1,p }}  ds . 
\end{align*}
Next integrating the second inequality of \eqref{Memo007-1}
over $\Omega \times (0,t )$, we get 
\begin{align*}
&\frac{1}{h} \int_{t -h }^{t} \int _{\Omega }  j ^{\ast } (\beta  (u (x, s ))) dxds 
	- \int _{\Omega }   j ^{\ast } (\beta  (u_ 0  (x ))) dx \\
& \geq \frac{1}{h} 
\int_{0}^{t} \int _{\Omega }  
	( \beta  (u (x, s )) - \beta  (u (x, s -h  ))) u (x ,s -h  )  dxds  \\
& = \frac{1}{h} 
\int_{-h }^{t-h } \int _{\Omega }  
	( \beta  (u (x, s+h  )) - \beta  (u (x, s  ))) u (x ,s  )  dxds
\end{align*}
By replacing $\beta  (u )$ with $\xi $, we have 
\begin{align*}
&\frac{1}{h} \int_{t -h }^{t} \int _{\Omega }  j ^{\ast } (\xi (x, s )) dxds 
	- \int _{\Omega }   j ^{\ast } (\beta  (u_ 0  (x ))) dx \\
& \geq  \frac{1}{h} 
\int_{-h  }^{t-h } \int _{\Omega }  
	( \xi (x, s+h  ) - \xi (x, s  ) ) u (x ,s  )  dxds  \\
& \geq  
\int_{ -h  }^{t-h } \int _{\Omega }  
	\LA  \frac{\xi (x, s+h  ) - \xi (x, s  )}{h }   ,  u (x ,s  )  \RA _{W ^{1, p }}ds  .
\end{align*}
Hence
\begin{align*}
&\liminf _{ h\to 0} \frac{1}{h} \int_{t -h }^{t} \int _{\Omega }  j ^{\ast } (\xi (x, s ))) dxds 
	- \int _{\Omega }   j ^{\ast } (\beta  (u_ 0  (x ))) dx \\
&\geq 
\int_{0}^{t} 
	\LA  \partial _t  \xi ( s   )  , u ( s   ) \RA  _{W ^{1,p }}  ds  , 
\end{align*}
and then the limit of $\frac{1}{h} \int_{t -h }^{t} \int _{\Omega }  j ^{\ast } (\xi (x, s ))) dxds$ as $h \to + 0 $ exists 
and satisfies 
\begin{equation}
\begin{split}
& \lim _{ h\to 0} \frac{1}{h} \int_{t -h }^{t} \int _{\Omega }  j ^{\ast } (\xi (x, s ))) dxds  \\ 
&=
\int _{\Omega }   j ^{\ast } (\beta  (u_ 0  (x ))) dx
+\int_{0}^{t} 
	\LA  \partial _t  \xi ( s   )  , u ( s   ) \RA  _{W ^{1,p }}  ds  .
\end{split}
\label{Pre003} 
\end{equation}

Similarly, by \eqref{Pre001}
\begin{equation}
\begin{split}
&j ^{\ast } (\beta  (u (x, s +h ))) - j ^{\ast } (\beta  (u (x, s  )))  \\
&\hspace{3cm} \leq 
		 ( \beta  (u (x, s+h )) - \beta  (u (x, s   ))) u (x ,s +h ) , \\
&j ^{\ast } (\beta  (u (x, s+h  ))) - j ^{\ast } (\beta  (u (x, s  ))) \\
&\hspace{3cm} \geq 
		 ( \beta  (u (x, s+h  )) - \beta  (u (x, s   ))) u (x ,s   ) .
\end{split}
\label{Memo007-2} 
\end{equation}
Integrating the first inequality of \eqref{Memo007-2} over $\Omega \times (-h ,t )$, we get 
\begin{align*}
&\int_{-h}^{t} \int _{\Omega }  j ^{\ast } (\beta  (u (x, s+h  ))) dxds 
	-\int_{-h }^{t} \int _{\Omega }    j ^{\ast } (\beta  (u (x, s  ))) dx ds  \\
&	=
\int_{t  }^{t+h } \int _{\Omega }  j ^{\ast } (\beta  (u (x, s ))) dxds 
	-h   j ^{\ast } (\beta  (u_ 0  (x ))) dx
\end{align*}
and 
\begin{align*}
&\frac{1}{h} \int_{t  }^{t+h } \int _{\Omega }  j ^{\ast } ( \xi (x, s )) dxds 
	- \int _{\Omega } j ^{\ast } (\beta  (u_ 0  (x )) dx \\
& \leq 
	\frac{1}{h} \int_{-h   }^{t } \int _{\Omega }  
	( \xi (x, s+h  ) - \xi (x, s   )) u (x ,s + h  )  dxds  \\
& = \frac{1}{h} 
\int_{0 }^{t + h } \int _{\Omega }  
		( \xi (x, s  ) - \xi  (x, s -h  )) u (x ,s  )  dxds  \\
& = 
- \int_{0 }^{t + h }  
		\LA  \frac{  \xi (x, s -h  ) - \xi (x, s  ) }{h} ,  u (x ,s  ) \RA _{W ^{1,p }}  ds  .
\end{align*}
Hence 
\begin{align*}
&\limsup _{h\to 0 } \frac{1}{h} \int_{t  }^{t+h } \int _{\Omega }  j ^{\ast } (\xi (x, s )) dxds 
	-  j ^{\ast } (\beta  (u_ 0  (x ))) dx \\
& \leq 
- \int_{0 }^{t  } 
		\LA  - \partial _t \xi (s    )  ,  u (  s  ) \RA _{W ^{1,p }}  ds  
 \leq 
 \int_{0 }^{t  } 
		\LA  \partial _t \xi (s   )   ,  u (  s  ) \RA _{W ^{1,p }}  ds  .
\end{align*}
By the same argument, we can derive from the second inequality of \eqref{Memo007-2}
\begin{align*}
&\liminf  _{h\to 0 } \frac{1}{h} \int_{t  }^{t+h } \int _{\Omega }  j ^{\ast } (\xi (x, s )) dxds 
	-  j ^{\ast } (\beta  (u_ 0  (x ))) dx \\
& \geq 
 \int_{0 }^{t  } 
		\LA  \partial _t \xi (s   )   ,  u (  s  ) \RA _{W ^{1,p }}  ds  .
\end{align*}
Therefore 
\begin{equation}
\begin{split}
&\lim _{ h\to 0} \frac{1}{h} 
\int_{t }^{t +h } \int _{\Omega }  j ^{\ast } (\xi (x, s ))) dxds  \\
&	=
 j ^{\ast } (\beta  (u_ 0  (x ))) dx
+\int_{0}^{t} 
	\LA  \partial _t  \xi ( s   )  , u ( s   ) \RA  _{W ^{1,p }}  ds .  
\end{split}
\label{Pre004} 
\end{equation}
By \eqref{Pre003} and \eqref{Pre004}, 
$\int_{0 }^{t } \int _{\Omega }  j ^{\ast } (\xi (x, s ))) dxds $
is differentiable for every $t > 0 $ and the assertion holds for every Lebesgue point of 
$t \mapsto \int _{\Omega }  j ^{\ast } (\xi (x, t ))) dx$.
\end{proof}

The main term $- \nabla  \cdot \alpha (\cdot , \nabla u) $ also can be represented by some subdifferential operator 
from $L ^p (0, T  ; W ^{1, p } _ 0 (\Omega ))$ onto  its dual space  $L ^{p'} (0, T  ; W ^{-1, p' } (\Omega ))$.
Define
\begin{equation*}
\varphi ( u ) :=  \int _{\Omega } a (x , \nabla v (x  )) dx ,
\end{equation*}
which satisfies $\varphi ( v ) < + \infty $ for every $v \in  W ^{1, p } _ 0  (\Omega ) $
by \eqref{A01}.
We can see that its subgradient at $ u \in  W ^{1, p } _ 0  (\Omega )  $
coincides with $- \nabla  \cdot \alpha (\cdot , \nabla u) \in  W ^{-1, p' } (\Omega )$.
Namely we have 
\begin{equation}
\begin{split}
\LA - \nabla  \cdot \alpha (\cdot , \nabla u) , v - u  \RA _{W ^{1,p } }
& =\int _{\Omega } \alpha (x , \nabla u (x )) \cdot ( \nabla v (x ) - \nabla  u (x ) ) dx \\
& \leq \int_{\Omega }   a (x , \nabla v (x  )) dx - \int_{\Omega }   a (x , \nabla u (x  )) dx 
\end{split}
\label{Pre002} 
\end{equation}
for every $u , v \in W ^{1,p  } _ 0 (\Omega ) $. 
Similarly, by defining 
\begin{equation*}
\Phi ( u ) := \int_{0}^{T} \int _{\Omega } a (x , \nabla u (x , t )) dx dt  ,
\end{equation*}
on $ L ^p (0, T ; W ^{1, p } _ 0  (\Omega ))$, 
we can show that  its subgradient at $ u \in L ^p (0, T ; W ^{1, p } _ 0  (\Omega )) $
coincides with 
$- \nabla  \cdot \alpha (\cdot , \nabla u) \in L ^{p'} (0, T  ; W ^{-1, p' } (\Omega ))$, 
that is to say, 
$- \nabla  \cdot \alpha (\cdot , \cdot ) : L ^p (0, T ; W ^{1, p } _ 0  (\Omega ))
\to  L ^{p'} (0, T  ; W ^{-1, p' } (\Omega ))$
is a maximal monotone operator. 
Hence we can use the following Lemma (see, e.g., Lemma 1.2 of \cite{BCP}):
\begin{Le}
Let
 $\{ u_ n \} _{ n \in \N }$ weakly converge to $u$ in $L ^p (0 ,T ; W ^{1, p } _ 0 (\Omega ))$ 
and 
 $\{ - \nabla  \cdot \alpha (\cdot , \nabla  u _n ) \} _{ n \in \N }$ weakly converge to $\eta $ in $L ^{p' } (0 ,T ; W ^{- 1, p'  }  (\Omega ))$. 
If 
\begin{equation*}
\limsup _{n \to \infty }\int_{0}^{T}  \LA - \nabla  \cdot \alpha (\cdot , \nabla   u _n  ) - \eta ,  u  _n - u  \RA _{W ^{1, p }} dt \leq 0 ,
\end{equation*}
then $\eta = - \nabla  \cdot \alpha (\cdot , \nabla u) $. 
\label{Demi}
\end{Le}

We here recall Aubin-Lions's compactness theorem for later use 
(see, e.g., Corollary 4 of \cite{Simon} or Theorem II.5.16 of \cite{BF}).
\begin{Le}
Let $ X_0$, $X$, $X_1$ be Banach spaces such that the embedding $ X_ 0 \hookrightarrow X$ is compact  and 
$X \hookrightarrow X_1$ is continuous. 
Define with $1 \leq q, r \leq \infty $
\begin{equation*}
W := \{  w \in L ^q (0, T ; X_ 0 ) ;~ \partial _ t w \in L ^r (0, T ; X_ 1 )   \} . 
\end{equation*}
Then  $W \hookrightarrow L ^q ( 0, T ; X )$ is compact if $q < \infty $ and  
 $W \hookrightarrow C ( [0, T ] ; X )$ is compact if $q = \infty $. 
\label{A-L-Th}
\end{Le}
In order to use this theorem in our proof, 
we prepare the following lemma. 
\begin{Le}
 Let $ p > 1$. Then for every $q \geq  1$, 
$W ^{1, p} (\Omega ) \cap  L^ {\infty } (\Omega )$ is compactly embedded in 
$L ^q (\Omega )$.
\label{Embed} 
\end{Le}
\begin{proof}
Let $\{ v_k \} _{k\in \N }$ be a sequence which is uniformly bounded in  $W ^{1, p} (\Omega ) \cap  L^ {\infty } (\Omega )$.
Then there exists a subsequence $\{ v_{k _ j } \} _{j\in \N }$
which strongly converges to some $v $ in  $L ^p (\Omega )$ by Rellich's compactness theorem. 
Moreover, we can extract a subsequence $\{ v_{k _ {j _ l } } \} _{l \in \N }$
which converges to $v$ for a.e. $x \in \Omega $. 
Since $v_{k _ {j _ l } } $ is uniformly bounded in $ L^ {\infty } (\Omega )$ and $\Omega $ is a bounded domain, 
we can show that $ v_{k _ {j _ l } } \to v $ strongly in $L ^q (\Omega )$ for every $q \geq 1 $
by Lebesgue's dominant convergence theorem. 
\end{proof}

\section{Proofs of Main Theorems}

In this paper, we consider the following time-discretization of \eqref{P}:
\begin{equation}
\begin{cases}
\DS ~~\frac{ \beta ( u ^{n+1} _ \tau (x ) ) - \beta (u ^{n} _ \tau   (x ) ) }{\tau } 
- \nabla \cdot \alpha ( x , \nabla    u ^{n+1} _ \tau  (x ) ) 
				 = F ^{M} (  \beta ( u ^{n} _ \tau (x )) ) + f ^n _{\tau }  (x ) ~~&~~x \in \Omega , \\[3mm]
\DS ~~ u ^0 _\tau (x) = u _ 0 (x ) ~~&~~ x \in \Omega ,
\end{cases}
\label{P001}
\end{equation}
where $ \tau = T / N $ ($ N \in \N $), and 
\begin{equation}
F ^{M} (s )
= \begin{cases}
~~ M ~~&~~\text{if } F (s) > M , \\
~~ F (s )  ~~&~~\text{if } - M \leq F (s) \leq M , \\
~~ -M ~~&~~\text{if } F (s) < -  M .
\end{cases}
\label{P002} 
\end{equation}
Moreover, let $f ^n _\tau (x) : = \frac{1}{\tau} \int_{n\tau }^{(n+1) \tau } f (x,t ) dt  $.

By Theorem 3.1 of \cite{U}, we can assure that
for given  $ u ^ {n} _{\tau } $
such that 
$ \beta (u ^ {n} _{\tau } )  \in  L ^\infty (\Omega )$, 
\eqref{P001} 
has a unique solution 
$ u ^ {n+1} _{\tau } \in W ^{1 , p } _ 0 (\Omega ) $
which satisfies 
$ \beta (u ^ {n + 1 } _{\tau } )  \in  L ^\infty (\Omega )$.
Then for any given $u _ 0 $ with $\beta (u _ 0 ) \in L ^{\infty } (\Omega ) $,
we obtain the sequence of solutions $u _{\tau } : = \{ u ^0 _{\tau } , u ^1 _{\tau } , \ldots , u ^N _{\tau } \}$
when (H.$\alpha $) and (H.$\beta $) are fulfilled.
Moreover, 
Theorem 3.1 of \cite{U} also implies that 
the solution $u ^{n+1} _{\tau } $  to \eqref{P001} satisfies 
\begin{equation}
\begin{split}
\| \beta (u ^{n+1 } _{\tau }) \| _{L ^{\infty } (\Omega )} 
&\leq \|  
 \beta (u ^{n} _ \tau ) + \tau   F ^{M} ( \beta (  u ^{n} _ \tau )  ) + \tau f ^n _{\tau } \|  _{L ^{\infty } (\Omega )} \\
&\leq \|  
  \beta (u ^{n} _ \tau ) \| _{L ^{\infty } (\Omega ) }
		+ \tau \|   F ^{M} (  \beta (  u ^{n} _ \tau ) ) \| _{L ^{\infty } (\Omega ) } +\tau  \|  f ^n _{\tau } \| _{L ^{\infty }  (\Omega ) } \\
&\leq \|  
  \beta (u ^{n} _ \tau ) \| _{L ^{\infty } (\Omega ) }+ \tau ( M  + \|  f  \| _{L ^{\infty } (Q ) } )  \\
\end{split}
\label{P003} 
\end{equation}
(in this proof, we have to assume $\beta (0 ) = 0 $). 
Summing \eqref{P003}, we obtain 
\begin{equation}
\max _{n=1 , \ldots, N }\| \beta ( u ^{n} _ \tau ) \| _{L^{\infty } (\Omega )} 
\leq \| \beta ( u _0 ) \| _{L^{\infty } (\Omega )} 
+ T ( M + \| f \| _{L^{\infty } (Q )  } ), 
\label{P004} 
\end{equation}
since $ \tau N = T $. 
Here, for given sequence $w _{\tau } : = \{ w ^0 _{\tau } , w ^1 _{\tau } , \ldots , w ^N _{\tau } \}$,
we define
$\Pi _{\tau } w _ \tau , \Lambda  _{\tau } w _ \tau : [0 , T ] \times \Omega \to \R $ by 
\begin{align*}
\Pi _{\tau } w _ \tau (t) 
& := \begin{cases}
~~w ^{n+1} _{\tau } ~~&~~ \text{ if } t \in ( n \tau , (n+1) \tau ]  , \\
~~w ^{0} _{\tau } ~~&~~ \text{ if } t= 0 ,
\end{cases} \\[3mm]
\Lambda  _{\tau } w _ \tau (t) 
& := \begin{cases}
~~\DS \frac{w ^{n+1} _{\tau } - w ^{n} _{\tau }  }{\tau } (t - n\tau ) + w ^{n} _{\tau }  ~~&~~ \text{ if } t \in ( n \tau , (n+1) \tau ]  , \\[3mm]
~~w ^{0} _{\tau } ~~&~~ \text{ if } t= 0 .
\end{cases}
\end{align*}
Then \eqref{P004} implies
\begin{equation}
\| \Pi  _{\tau } \beta ( u ^{n} _ \tau ) \| _{L^{\infty } ( Q )} 
\leq \| \beta ( u _0 ) \| _{L^{\infty } (\Omega )} 
+ T ( M + \| f \| _{L^{\infty }   } (Q )) . 
\label{P005} 
\end{equation}

Multiplying \eqref{P001} by  $ u ^ { n +1 } _{\tau } $ and using \eqref{Pre001}, \eqref{Pre002}, and \eqref{A02}, we have 
\begin{align*}
&\frac{ 1 }{\tau } \LC \int_{\Omega } j ^{\ast} (\beta (  u ^{n+1} _ \tau  ))dx -  \int_{\Omega } j ^{\ast} (\beta (  u ^{n} _ \tau  ))dx  \RC + 
\int _{\Omega } a ( x ,  \nabla  u ^{n+1} _ \tau ) dx  \\
\leq &  
\int _{\Omega }
( F ^{M} (  \beta ( u ^{n} _ \tau ) ) + f ^n _{\tau }  )  u ^{n+1} _ \tau  dx \\
\leq &  
C( M + \| f \| _{L ^{\infty } (Q) } ) \| \nabla  u ^{n+1} _ \tau  \|  _{L ^p (\Omega ) } ,
\end{align*}
where and henceforth, 
$c $ and $C$ denotes the positive general constant which is independent of $N , \tau $. 
By $\tau N =T $ and \eqref{A01}, we obtain 
\begin{equation}
\max _{n = 1 ,\ldots , N } \int_{\Omega } j ^{\ast} (\beta (  u ^{n+1} _ \tau  ))dx
+ \tau \sum_{n=1}^{N}  \| \nabla  u ^{n+1} _ \tau \| ^p _{L^p (\Omega )} \leq C , 
\label{P006} 
\end{equation}
which implies 
\begin{equation}
\int_{0}^{T}  \| \nabla  \Pi  _{\tau } u  _ \tau (t) \| ^p _{L^p (\Omega )} dt \leq C .
\label{P007} 
\end{equation}

We next derive the $L ^{\infty}$-estimate of $u_{\tau }$. 
Let
\begin{equation*}
 K ^r _L (s) :=
 \begin{cases}
~~|s| ^ {r -2} s ~~&~~ |s| \leq L, \\
~~L \sgn s  ~~&~~ |s| > L ,
\end{cases} 
\end{equation*}
with $r >2 $ and $L > 0 $, where $\sgn s := s  / |s | $ is the sign function. 
We multiply \eqref{P001} by $K ^{p (r -1) + 2 } _L(  u ^{n+1} _ \tau  )$.
Since $ K ^{p (r -1) + 2 } _L $ is monotone increasing,
\begin{equation*}
\Gamma (s) = \begin{cases}
~~\DS \int_{0}^{s } K ^{p (r -1) + 2 } _L ( (\beta ^{-1 } )^{\circ} (\sigma ) ) d \sigma ~~&~~\text{if }  s \in 
		\overline{D(K ^{p (r -1) + 2 } _L \circ \beta ^{-1 } )} 
		= \overline{ D( \beta ^{-1 } ) }  , \\
~~+ \infty     ~~&~~ \text{otherwise, }  
\end{cases}
\end{equation*}
becomes a lower semi-continuous convex function and the subdifferential can be defined.
Let $s' \in K ^{p (r -1) + 2 }  \circ \beta ^{-1 } (s_ 0 )$.
If $s > s _ 0 $, 
\begin{equation*}
\Gamma (s) - \Gamma  (s_ 0 ) = \int_{s_0}^{s }  K ^{p (r -1) + 2 } _L ( (\beta ^{-1 } )^{\circ} (\sigma ) ) d \sigma
\geq s' (s  - s_ 0 )
\end{equation*}
by $ s'  \leq K ^{p (r -1) + 2 } _L ( (\beta ^{-1 } )^{\circ} (\sigma ) ) (s ) $
and if 
$s < s _ 0 $, 
\begin{equation*}
\Gamma (s) - \Gamma  (s_ 0 ) =- \int_{s}^{s_0  }  K ^{p (r -1) + 2 } _L ( (\beta ^{-1 } )^{\circ} (\sigma ) ) d \sigma
\geq - s' (s _ 0   - s )
=s' (s   - s _0 )
\end{equation*}
by $ K ^{p (r -1) + 2 } _L ( (\beta ^{-1 } )^{\circ} (\sigma ) ) (s ) \leq s' $.
It follows from this that 
the  subdifferential of $\Gamma $ satisfies 
$\partial \Gamma \supset  K ^{p (r -1) + 2 }  \circ \beta ^{-1 } $. 
Then by $K ^{p (r -1) + 2 } _L(  u ^{n+1} _ \tau  (x) ) \in  K ^{p (r -1) + 2 } _L( \beta ^{-1}  ( \beta ( u ^{n+1} _ \tau  (x )) )) $
for a.e. $x\in \Omega $, we have 
\begin{equation*}
\int_{\Omega }  (\beta ( u ^{n+1} _ \tau ) - \beta (u ^{n} _ \tau )  ) K ^{p (r -1) + 2 } _L(  u ^{n+1} _ \tau  ) dx 
 \geq 
\int_{\Omega }  \Gamma ( \beta ( u ^{n+1} _ \tau ) ) dx  -\int_{\Omega }    \Gamma ( \beta ( u ^{n} _ \tau ) ) dx .
\end{equation*}
The 2nd term of LHS becomes by \eqref{A03}
\begin{align*}
 & \int_{\Omega }  -\nabla  \cdot \alpha  (x , \nabla u ^{n+1} _ \tau  )   K ^{p (r -1) + 2 } _L(  u ^{n+1} _ \tau  ) dx \\
= & 
(p (r -1) + 1 ) \int_{ | u ^{n+1} _ \tau| \leq L  }  
 \alpha  (x , \nabla u ^{n+1} _ \tau  ) \cdot \nabla u ^{n+1} _ \tau    |  u ^{n+1} _ \tau  |  ^{p (r -1)}   dx \\
\geq   & 
c(p (r -1) + 1 ) \int_{ | u ^{n+1} _ \tau| \leq L  }  | \nabla u ^{n+1} _ \tau | ^{p}  |  u ^{n+1} _ \tau  |  ^{p (r -1)}   dx \\
\geq  & 
c \int_{ | u ^{n+1} _ \tau| \leq L  }  |  |  u ^{n+1} _ \tau  |  ^{ (r -1)} \nabla u ^{n+1} _ \tau | ^{p}    dx \\
= & 
\frac{ c }{r ^p } \int_{ | u ^{n+1} _ \tau| \leq L  }  |  \nabla  |  u ^{n+1} _ \tau  |  ^{ r} | ^{p}    dx .
\end{align*}
Moreover, 
\begin{equation*}
\int_{\Omega } ( F ^M ( u ^{n+1} _ \tau  ))  + f^{n} _{\tau } ) K ^{p (r -1) + 2 } _L (  u ^{n+1} _ \tau  ) dx 
\leq 
(M + \| f \| _{L^{\infty }} ) 
\| u ^{n+1} _ \tau   \|  ^{p (r -1) + 1 } _{L ^{p (r -1) + 1 } (\Omega ) } . 
\end{equation*}
Hence we obtain 
\begin{align*}
&\int_{\Omega }  \Gamma ( \beta ( u ^{n+1} _ \tau ) ) dx  -\int_{\Omega }    \Gamma ( \beta ( u ^{n} _ \tau ) ) dx 
+
\frac{c \tau }{r ^p }   \int_{ | u ^{n+1} _ \tau| \leq L  }  |  \nabla  |  u ^{n+1} _ \tau  |  ^{ r} | ^{p}    dx \\
&~~~~~\leq (M + \| f \| _{L^{\infty } (Q) } ) 
\tau   \| u ^{n+1} _ \tau   \|  ^{p (r -1) + 1 } _{L ^{p (r -1) + 1 } (\Omega )} . 
\end{align*}
Summing this from  
$n = 0 $ to $n=  N -1 $ and taking the limit $L \to \infty$, we get 
\begin{align*}
&\int_{\Omega }  \Gamma ( \beta ( u ^{N} _ \tau ) ) dx  -\int_{\Omega }    \Gamma ( \beta ( u _0 ) ) dx 
+
\frac{ c }{r ^p }  \sum_{ n=0 }^{N -1 } \tau  \int_{ \Omega   }  |  \nabla  |  u ^{n+1} _ \tau  |  ^{ r} | ^{p}    dx \\ 
&~~~~~~~~\leq (M + \| f \| _{L^{\infty } (Q) } ) 
 \sum_{ n=0 }^{N -1 } \tau   \| u ^{n+1} _ \tau   \|  ^{p (r -1) + 1 } _{L ^{p (r -1) + 1 }} , 
\end{align*}
which implies 
\begin{equation}
\begin{split}
&\int_{\Omega }  \Gamma ( \beta ( u ^{N} _ \tau ) ) dx  +
\frac{ c }{r ^p }  \int_{0}^{T}   \int_{ \Omega   }  |  \nabla  | \Pi _{\tau } u_ \tau  |  ^{ r} | ^{p}    dx dt \\ 
&~~~~~~~~\leq (M + \| f \| _{L^{\infty } (Q) } ) 
  \| \Pi _{\tau }u ^{n+1} _ \tau   \|  ^{p (r -1) + 1 } _{L ^{p (r -1) + 1 } (Q ) } +
\int_{\Omega }    \Gamma ( \beta ( u _0 ) ) dx .
\end{split}
\label{P008} 
\end{equation}

We note that $\Gamma  \geq 0 $. Moreover,
we have by the  definition of $\Gamma $
\begin{align*}
& 
\Gamma ( \beta ( u _0 ) ) - \Gamma (\beta (0  ) )  \leq (\beta ( u _0 )  - \beta ( 0 ) ) K ^{p (r -1) + 2 } (\beta ^{-1} (\beta (u_0 )))  \\
\Rightarrow ~~ &     
\Gamma ( \beta ( u _0 ) ) \leq | \beta ( u _0 ) | | K ^{p (r -1) + 2 } (u_0 ) |  
\leq | \beta ( u _0 ) | | u_0 |  ^{p (r -1) + 1 }   
\end{align*}
for a.e. $x\in \Omega $.
Then \eqref{P008} yields 
\begin{align*}
&\frac{ c }{r ^p }  \int_{0}^{T}   \int_{ \Omega   }  |  \nabla  | \Pi _{\tau } u_ \tau  |  ^{ r} | ^{p}    dx dt \\ 
&~~~~~~~~\leq (M + \| f \| _{L^{\infty } (Q)} ) 
  \| \Pi _{\tau }u ^{n+1} _ \tau   \|  ^{p (r -1) + 1 } _{L ^{p (r -1) + 1 } (Q ) } +
\| \beta ( u _0 ) \| _{L^{\infty } (\Omega )}
\|  u _0  \|   ^{p (r -1) + 1 } _{L  ^{p (r -1) + 1 } (\Omega )}.
\end{align*}
By the Sobolev inequality, there exists some $\mu > 1 $ such that
\begin{equation*}
\frac{ c }{r ^p }   \|  \Pi _{\tau } u _ \tau   \| ^{rp }_{L ^{\mu rp } (Q )} 
\leq (M + \| f \| _{L^{\infty }(Q)} ) 
 \|  \Pi _{\tau } u _ \tau   \|  ^{p (r -1) + 1 } _{L ^{p (r -1) + 1 } (Q )}  
+
\| \beta ( u _0 ) \| _{L^{\infty } (\Omega )}
\|  u _0  \|   ^{p (r -1) + 1 } _{L  ^{p (r -1) + 1 } (\Omega )}
\end{equation*}
Applying Moser's iteration method, i.e., by letting 
$r = \mu ^{l} $ ($l= 0 , 1 , \ldots $) and taking the limit $ l \to \infty $, we obtain 
\begin{equation}
\| \Pi _{\tau } u _ \tau   \| _{ L ^{\infty } (Q )} \leq C ( \| f\|  _{L ^{\infty } (Q)} , \| u _0 \| _{L^{\infty} (\Omega ) } , 
\| \beta ( u _0 )  \| _{L^{\infty} (\Omega ) } , |\Omega | , |Q| , M  ) . 
\label{P009} 
\end{equation}
In \eqref{P005}  and \eqref{P009}, 
we can derive the uniform boundedness of $\Pi _{\tau } u _ \tau  $ and $\Pi _{\tau } \beta (u _ \tau )   $. 
From this, 
we can deal with  $\beta $ and $ \beta ^{-1}$ as Lipschitz continuous functions 
by  the assumption of local Lipschitz continuity in the theorem.


\subsection{Proof of Theorem \ref{Th01}}

In addition to the estimates given above, we consider the case where 
$\beta ^{-1}$ is locally Lipschitz continuous and derive additional estimates of solutions.  
We exploit the method of  \cite{Ba}. 
Multiply \eqref{P001} by $ u ^{n+1} _{\tau } -   u ^{n} _{\tau } $, we get 
\begin{align*}
& \frac{ 1 }{\tau }
\int _{\Omega }( \beta ( u ^{n+1} _ \tau ) - \beta (u ^{n} _ \tau ) ) ( u ^{n+1} _ \tau  - u ^{n} _ \tau )dx 
+
\int _{\Omega  } a (x , \nabla   u ^{n+1} _ \tau (x)  ) dx  - \int _{\Omega  } a (x , \nabla   u ^{n} _ \tau (x)  ) dx  \\
& \hspace{3cm} \leq  
(M + \| f \| _{L^{\infty } (Q) }) |\Omega | ^{1/2 }
\|   u ^{n+1} _ \tau  - u ^{n} _ \tau \| _{L ^2 (\Omega )} .  
\end{align*}
Since $\Pi _{\tau } u _ \tau  $ and $\Pi _{\tau } \beta (u _ \tau )   $ is uniformly bounded 
with respect to the $L ^{\infty}$-norm, 
we can deal with  $ \beta ^{-1}$ as Lipschitz continuous functions.  
Hence  we have 
\begin{align*}
 | \beta ( u ^{n+1} _ \tau (x )) - \beta (u ^{n} _ \tau  (x )) |
&\geq 
c 
 | \beta ^{-1} ( \beta ( u ^{n+1} _ \tau ) (x) )-  \beta ^{-1} (  \beta (u ^{n} _ \tau (x )) ) | \\
&= 
c |  u ^{n+1} _ \tau (x)  -  u ^{n} _ \tau  (x ) |
\end{align*}
for a.e. $x \in \Omega $ and then 
\begin{align*}
& \frac{ c }{\tau }
\|  u ^{n+1} _ \tau  - u ^{n} _ \tau \| ^2 _{L^2 (\Omega ) } + 
\int _{\Omega  } a (x , \nabla   u ^{n+1} _ \tau (x)  ) dx  - \int _{\Omega  } a (x , \nabla   u ^{n} _ \tau (x)  ) dx \\
& \leq  
(M + \| f \| _{L^{\infty }}) |\Omega | ^{1/2 }
\|   u ^{n+1} _ \tau  - u ^{n} _ \tau \| _{L ^2 (\Omega )} \\
& \leq 
\tau C + 
\frac{c }{2\tau } \|   u ^{n+1} _ \tau  - u ^{n} _ \tau \| ^2 _{L^2 } .
\end{align*}
This yields 
\begin{equation*}
c \sum_{n=0}^{N-1 } \tau 
\LN  \frac{ u ^{n+1} _ \tau  - u ^{n} _ \tau}{ \tau } \RN ^2 _{L^2 } + 
\max _{n= 1 , \ldots , N }\int _{\Omega  } a (x , \nabla   u ^{n} _ \tau (x)  ) dx
\leq 
T  C + 
\int _{\Omega  } a (x , \nabla   u  _ 0  (x)  ) dx . 
\end{equation*}
By \eqref{A01}, 
\begin{equation}
\| \partial   _ t  \Lambda  _ \tau u _{\tau }\| ^2_{L^2 (Q )}
+ 
\sup _{ 0\leq t \leq T } \| \nabla \Pi _{\tau } u _{\tau } (t) \| ^p _{L^p ( \Omega )} \leq 
C . 
\label{P010} 
\end{equation}
If  $ t \in ( n \tau , (n+1) \tau ]$, we get 
\begin{align*}
\| \nabla \Lambda  _{\tau } u _{\tau } (t) \| _{L^p (\Omega )} 
&= \LN 
\frac{   t - n\tau }{\tau } \nabla u  ^{n+1} _{\tau } 
+
\frac{  (n +1  ) \tau -t  }{\tau }  \nabla  u  ^{n} _{\tau }
\RN  _{L^p (\Omega )} \\
&\leq 
 \LN 
 \nabla u  ^{n+1} _{\tau } \RN  _{L^p (\Omega )} 
+
\LN \nabla  u  ^{n} _{\tau }
\RN  _{L^p (\Omega )} ,
\end{align*}
namely, 
\begin{equation}
\sup _{ 0\leq t \leq T } \| \nabla \Lambda  _{\tau } u _{\tau } (t) \| ^p _{L^p ( \Omega )} \leq 
C . 
\label{P011} 
\end{equation}
Similarly, \eqref{P009} yields 
\begin{equation}
\| \Lambda  _{\tau } u _{\tau } (t) \|  _{L^\infty  ( Q  )} \leq 
C . 
\label{P012} 
\end{equation}

According to \eqref{P005}, there exists some subsequence of $\{ \Pi _{\tau } \beta (u _\tau ) \} _{\tau >0 }$
(we omit relabeling) such that 
\begin{equation*}
\Pi _\tau \beta (u _\tau ) \rightharpoonup  \xi ~~~~~~\ast\text{-weakly in } L^\infty (Q) 
\end{equation*}
as $\tau \to 0 $ and its limit $\xi \in L ^{\infty } (Q )$ satisfies 
\begin{equation}
\| \xi \| _{L ^\infty (Q )} \leq \| \beta ( u _0 ) \| _{L^{\infty } (\Omega )} 
+ T ( M + \| f \| _{L^{\infty }   } (Q )) . 
\label{P013} 
\end{equation}
Moreover, \eqref{P010},  
\eqref{P011}, and  
\eqref{P012} leads to 
\begin{align*}
&\Lambda  _\tau  u _\tau  \rightharpoonup  u ~~&\ast\text{-weakly in } L^\infty (0,T ; W^{1, p } _ 0  (\Omega )) \cap 
		L^{\infty } (Q ) , \\
&\partial _t \Lambda  _\tau  u _\tau  \rightharpoonup \partial _t u ~~~&\text{weakly in } L^2  (0,T ; L^2   (\Omega )) .
\end{align*}
with some $u\in L^\infty (0,T ; W^{1, p } _ 0  (\Omega )) \cap 
		L^{\infty } (Q )\cap W ^{1, 2 }   (0,T ; L^2   (\Omega )) $.
By \eqref{A02},
\begin{align*}
\LN - \nabla \cdot \alpha (\cdot , \nabla  \Pi _{\tau }u _{\tau } (t ))\RN _{W ^{-1 , p ' } (\Omega )}
&=\sup _{\| \phi \| _{W ^{1, p } _ 0  (\Omega )} =1 }
 \LA - \nabla \cdot \alpha (\cdot  ,\nabla \Pi _{\tau } u _{\tau } (\cdot , t )) , \phi  \RA _{W ^{1,p }} \\
 &\leq \sup _{\| \phi \| _{W ^{1, p } _ 0  (\Omega )} =1 } 
 \int_{\Omega } | \alpha (x , \nabla  \Pi _{\tau }u _{\tau } (x, t )) | |\nabla  \phi (x ) |  dx \\
 &\leq 
\LC \int_{\Omega } | \alpha (x , \nabla \Pi _{\tau } u _{\tau } (x, t )) | ^{p' } dx \RC ^{1/p' } \\
&\leq 
\LC \int_{\Omega } C ( | \nabla \Pi _{\tau } u _{\tau } (x, t ) | ^p +1 ) dx \RC ^{1/p' } .
\end{align*}
Then from \eqref{P007},  we can derive 
\begin{equation*}
\int_{0}^{T} \LN - \nabla \cdot \alpha (\cdot , \nabla \Pi _{\tau } u _{\tau } (t ))\RN ^{p'}_{W ^{-1 , p ' } (\Omega )}dt 
\leq C 
\end{equation*}
and then there exists some $ \eta \in  L^{p'} (0,T ; W^{-1 , p' } (\Omega ))$ such that
\begin{equation*}
- \nabla \cdot \alpha (\cdot , \nabla \Pi _{\tau } u _{\tau } (t ))  \rightharpoonup  \eta 
 ~~~~~~\text{weakly in } L^{p'} (0,T ; W^{-1 , p' } (\Omega )) .
\end{equation*}
Furthermore, since 
$\partial _{t} \Lambda _\tau  \beta (u _\tau ) (t)
=
 \frac{ \beta (u ^{n+1}_\tau ) -  \beta (u^{n} _\tau )}{\tau } 
 $, 
 we can see that from the equation \eqref{P001}
\begin{equation*}
\partial _{t} \Lambda _\tau  \beta (u _\tau )
 \rightharpoonup  \partial _t \xi   ~~~~~~\text{weakly in } L^{p'} (0,T ; W^{-1 , p' } (\Omega )) .
\end{equation*}

By applying Lemma \ref{A-L-Th} with $X_0 = W ^{1, p } (\Omega ) \cap L ^{\infty } (\Omega )$,
$X = X_1 = L ^2 (\Omega )$
(note that the embedding $X _0 \hookrightarrow X =X_1 $ is compact by Lemma \ref{Embed}), 
we can see 
from \eqref{P010}, \eqref{P011}, and  
\eqref{P012} that there exists some subsequence of $\{ \Lambda  _\tau  u _\tau \}  _{\tau >  0 }$
such that 
\begin{equation*}
\Lambda  _\tau  u _\tau  \to u ~~~~~~\text{strongly in } C ([ 0,T] ; L^2   (\Omega )) . 
\end{equation*}
Here when $t \in (n\tau , (n+1 ) \tau ]$, 
\begin{align*}
\|  \Pi _{\tau } u _\tau (t) -  \Lambda  _{\tau } u _\tau (t)  \| _{L^2 (\Omega )}
=& \LN  u_\tau - \frac{u ^{n+1}_{\tau } - u ^{n}_{\tau } }{\tau } (t - n\tau ) -u ^{n}_{\tau }   \RN _{ L^2 (\Omega )} \\
=& ( (n+1) \tau -t ) \LN  \frac{u ^{n+1}_{\tau } - u ^{n}_{\tau } }{\tau }    \RN _{L^2 (\Omega ) } , 
\end{align*}
which leads to 
\begin{equation*}
\int_{n \tau } ^{ (n+1) \tau } 
\|  \Pi _{\tau } u _\tau (t) -  \Lambda  _{\tau } u _\tau (t)  \| ^{2}_{L^2 (\Omega )} dt
= 
\frac{1}{3} \tau ^{2 +1} \LN    \frac{u ^{n+1}_{\tau } - u ^{n}_{\tau } }{\tau }   \RN ^{2 } _{ L ^2 (\Omega ) } .
\end{equation*}
Hence we have by $ \sum_{n=0}^{ N -1 } $
\begin{equation*}
\int_{0 } ^{ T } 
\|  \Pi _{\tau } u _\tau (t) -  \Lambda  _{\tau } u _\tau (t)  \| ^{2 }_{L^ 2 } dt 
= 
\frac{1}{3} \tau ^{2 } \int_{0}^{T } \LN  \partial _t \Lambda _\tau u _{\tau }   (t)  \RN ^{2 } _{L^2 } dt
\leq C \tau ^2  ,  
\end{equation*}
which immediately implies 
\begin{equation*}
\Pi _\tau  u _\tau  \to u ~~~~~~\text{strongly in } 	L^2 ( 0,T ; L^2   (\Omega )) .
\end{equation*}
Therefore, we can extract some subsequence of $\{ \Pi  _\tau  u _\tau \}  _{\tau >  0 }$
which satisfies 
$\Pi _\tau  u _\tau (x, t ) \to u (x,t )$ a.e. $ \Omega \times (0,T )$. 
Since $\beta $ is a continuous map, 
we obtain 
$\Pi _\tau \beta ( u _\tau  ) \to \beta (u  ) $ for a.e. $ \Omega \times (0,T ) $. 
By Lebesgue's dominant convergence theorem 
\begin{equation*}
\Pi _{\tau } F ^M (\beta ( u_ \tau ))
\to F^ M (\beta (u ) )~~~~~~\text{strongly in } 	L^q ( 0,T ; L^q   (\Omega )) ~~\forall q \geq 1 .
\end{equation*}
By the dominant convergence theorem, we also have 
\begin{equation*}
\Pi _{\tau } f_ \tau  
\to f~~~~~~\text{strongly in } 	L^q ( 0,T ; L^q   (\Omega )) ~~\forall q \geq 1 .
\end{equation*}
Therefore \eqref{P001} weakly converges to 
$\partial _t  \xi + \eta = F ^M (\beta (u )) + f $ in 
$L ^{p'} (0,T ; W ^{-1, p'} (\Omega ))$.

According to 
$\Pi _\tau \beta ( u _\tau ) \to \beta (u ) $ for a.e. $ \Omega \times (0,T )$ and
 \eqref{P005},  we obtain 
\begin{equation*}
\Pi _{\tau } \beta ( u_\tau ) \to \beta ( u ) ~~~~\text{strongly in }   	L^q ( 0,T ; L^q   (\Omega )) ~~\forall q \geq 1 .
\end{equation*}
hence $\xi = \beta (u)$ for a.e. $ \Omega \times (0,T )$. 
We here show $\eta = - \nabla \cdot  \alpha  ( \cdot , \nabla  u )   $ in order to assure 
$u$ is a solution to 
$\partial _t  \beta (u) - \nabla \cdot  \alpha  ( \cdot , \nabla  u )  = F ^M (\beta (u )) + f $. 
Since 
$\Pi _ \tau \beta (u _\tau ) (T) = \beta (u ^N _ {\tau })$ is uniformly  bounded in $L^{\infty } (\Omega )$
from \eqref{P005},  
so there exists a subsequence which weakly converges in $L^2 (\Omega )$. 
Alternatively,  
$u ^N _{\tau } = \Lambda _{\tau } u _{\tau} (T)$ strongly converges to $u (T)$ in  $L^2 (\Omega )$
and 
we can extract a subsequence such that 
$\beta (u ^N _ {\tau }) \rightharpoonup  \beta (u (T)) $ weakly in $L^2$ as $\tau \to 0 $
by the demi-closedness of maximal monotone operator. 
Then since 
$w \mapsto \int_{\Omega } j ^{\ast} (w ) dx  $ is a 
lower semi-continuous convex function on $L^2 (\Omega )$,
\begin{equation*}
\liminf _{\tau \to 0 } \int_{\Omega } j ^{\ast} (\beta (u ^N _ {\tau })  ) dx \geq \int_{\Omega } j ^{\ast} (\beta (u (T))  ) dx 
\end{equation*}
holds. 
Moreover, 
$ \int_{\Omega } j ^{\ast} (\beta (u ^N _ {\tau })  )$ is uniformly bounded as $\tau \to 0 $
by \eqref{P006}
and 
then there exists a subsequence which converges to some number $\theta \in \R $. 
Hence we have  
\begin{equation}
\theta \geq \int_{\Omega } j ^{\ast} (\beta (u (T))  ) dx .
\label{P014}  
\end{equation}

Again, multiplying \eqref{P001} by $u ^{n+1} _{\tau } $, we get 
\begin{align*}
& \int _{\Omega } j ^{\ast} (\beta (u ^{n+1} _ {\tau }) ) dx  - \int _{\Omega } j ^{\ast} (\beta (u ^{n} _ {\tau }) ) dx
+\tau  \LA - \nabla  \cdot \alpha (\cdot , \nabla  u ^{n+1} _{\tau } ),  u ^{n+1} _{\tau } \RA _{W ^{1, p }} \\
&\leq 
\tau \int_{\Omega }  f ^n _{\tau }  u ^{n+1} _{\tau } dx 
+
\tau \int_{\Omega }  F ^M (\beta (u ^{n} _{\tau }))  u ^{n+1} _{\tau } dx .
\end{align*}
Summing this from  $n = 0 $ to $n=  N-1$, we have 
\begin{align*}
& \int _{\Omega } j ^{\ast} (\beta (u ^{N} _ {\tau }) ) dx  - \int _{\Omega } j ^{\ast} (\beta (u _0 ) ) dx
+\int_{0}^{T}  \LA - \nabla  \cdot \alpha (\cdot , \nabla \Pi _\tau  u _{\tau } ),  u  _{\tau } \RA _{W ^{1, p }} dt \\
&\leq 
\int_{0}^{T}  \int_{\Omega }  \Pi _{\tau } f _{\tau } \Pi _{\tau }   u  _{\tau } dx dt  
+
\int_{0}^{T} \int_{\Omega } \Pi _{\tau }   F ^M (\beta (u  _{\tau }))  \Pi _{\tau } u _{\tau } dx dt . 
\end{align*}
Taking the limit as $\tau \to 0 $, we obtain 
\begin{align*}
& \theta   - \int _{\Omega } j ^{\ast} (\beta (u _0 ) ) dx
+\limsup _{\tau \to 0 }\int_{0}^{T}  \LA - \nabla  \cdot \alpha (\cdot , \nabla \Pi _\tau  u _{\tau } ),  u  _{\tau } \RA _{W ^{1, p }} dt \\
&\leq 
\int_{0}^{T}  \int_{\Omega }   f u  dx dt  
+
\int_{0}^{T} \int_{\Omega }   F ^M (\beta (u  ))   u dx dt . 
\end{align*}
By \eqref{P014} and Lemma \ref{IBP}, 
\begin{align*}
&\limsup _{\tau \to 0 }\int_{0}^{T}  \LA  - \nabla  \cdot \alpha (\cdot , \nabla \Pi _\tau  u _{\tau } ),  u  _{\tau } \RA _{W ^{1, p }} dt \\
\leq &
\int_{0}^{T}  \int_{\Omega }   f u  dx dt  
+
\int_{0}^{T} \int_{\Omega }   F ^M (\beta (u  ))   u dx dt
-\int_{\Omega } j ^{\ast} (\beta (u (T))  ) dx
+ \int _{\Omega } j ^{\ast} (\beta (u _0 ) ) dx \\
\leq & 
\int_{0}^{T}  \LA   F ^M (\beta (u  ))  + f  -\partial _t \beta (u ) , u  \RA _{W ^{1,p  }} dx dt  
=
\int_{0}^{T}  \LA  \eta  , u  \RA _{W ^{1,p  }} dx dt  .
\end{align*}
Therefore we can derive 
\begin{equation*}
\limsup _{\tau \to 0 }\int_{0}^{T}  \LA - \nabla  \cdot \alpha (\cdot , \nabla \Pi _\tau  u _{\tau } ) - \eta ,  u  _{\tau } - u  \RA _{W ^{1, p }} dt \leq 0 
\end{equation*}
and we can assure 
$ \eta = - \nabla  \cdot \alpha (\cdot , \nabla  u )  $ by Lemma \ref{Demi}.

Finally, we show that the limit $u $ becomes a time local solution to \eqref{P}. 
By \eqref{P005}, 
$\|  \beta (u) \|  _{L^\infty ( ( 0, T' ) \times \Omega  )} \leq 2 \| \beta (u_ 0 ) \|  _{L^\infty (\Omega )} $
holds with some sufficiently small $0< T' \leq  T $. 
Thus by setting 
\begin{equation*}
M := \max _{|s| \leq 4  \| \beta (u_0 )\|  _{L^\infty (\Omega )} }|F (s)| , 
\end{equation*}
we can assure that 
$F ^M (\beta (u (x , t ))) =F (\beta (u (x ,t )))$ for a.e. $ \Omega \times (0, T ' ) $ and 
then $u $ is a solution to \eqref{P} while $(0, T ' )$.


\subsection{Proof of Theorem  \ref{Th02}} 

We next assume that  $\beta $ is locally Lipschitz continuous. 
Since 
$  \Pi _{\tau } u _ \tau  ,  \Pi _{\tau } \beta ( u _ \tau )   $ is uniformly bounded, 
we can deal with $\beta $ as a Lipschitz continuous function.
Then we have from \eqref{P007}
\begin{equation}
\int_{0}^{T} \| \nabla \Pi _ \tau \beta (  u _{ \tau } )  \| ^p _{L^p } dt 
\leq 
C \int_{0}^{T} \| \nabla  \Pi _ \tau u _{ \tau }   \| ^p _{L^p } dt 
 \leq C .
 \label{P015}
\end{equation}
Since \eqref{P005}, \eqref{P007}, \eqref{P009} hold regardless of the condition of $\beta $,
we can extract subsequences such that
\begin{align*}
&
\Lambda  _\tau \beta (u _\tau ) \rightharpoonup  \xi ~~~&\ast\text{-weakly in } L^\infty (Q) 
		\text{ and weakly in } L^p  (0,T ; W^{1, p } _ 0  (\Omega )) , \\
&
\Lambda  _\tau  u _\tau  \rightharpoonup  u ~~~~~&\ast\text{-weakly in }
		L^{\infty } (Q ) \text{ and weakly in } L^p  (0,T ; W^{1, p } _ 0  (\Omega )), \\
&
- \nabla  \cdot \alpha (\cdot , \nabla \Pi _\tau  u _{\tau } )  \rightharpoonup  \eta  ~~&\text{weakly in } L^{p'} (0,T ; W^{-1 , p' } (\Omega )) , \\ 
&
 \partial _{t} \Lambda _\tau  \beta (u _\tau )
 \rightharpoonup  \partial _t \xi   ~&\text{weakly in } L^{p'} (0,T ; W^{-1 , p' } (\Omega ))  ,
\end{align*}
with some 
\begin{align*}
& \xi \in L ^{\infty} (Q) \cap  L^p  ( 0,T ; W^{1, p } _ 0  (\Omega )) 
\cap  W^{1 , p'} (0,T ; W^{-1 , p' } (\Omega )) ,  \\
&u \in L ^{\infty} (Q) \cap  L^p  ( 0,T ; W^{1, p } _ 0  (\Omega ))  , \\
&\eta \in  L^{p'} (0,T ; W^{-1 , p' } (\Omega )) .
\end{align*}
By  Lemma \ref{A-L-Th} with 
$X _0 = W ^{1,p} (\Omega ) \cap L^{\infty } (\Omega )$, 
$X = L ^{p'} (\Omega )$, 
$X_1 = W ^{-1, p' } (\Omega )$, 
\begin{equation*}
\Lambda  _\tau \beta (u _\tau ) \to \xi ~~~~~\text{strongly in } L^p (0,T ; L^{p'} (\Omega )) . 
\end{equation*}
Then by extracting a subsequence such that 
$\Lambda  _\tau \beta (u _\tau ) \to \xi$ for 
a.e. $ \Omega \times (0,T )$, 
using \eqref{P005}, and applying the dominant convergence theorem, we can show that 
\begin{equation*}
\Lambda  _\tau \beta (u _\tau ) \to  \xi ~~~~~\text{strongly in } L^q (0,T ; L^{q} (\Omega ))
~~~\forall q > 1  . 
\end{equation*}
Lemma \ref{Demi-beta} implies that $\xi = \beta (u)$ a.e. $\Omega \times (0,T )$. 
By the same reasoning as our proof of Theorem \ref{Th01}, 
\begin{equation*}
\Pi _{\tau } F ^M (\beta ( u_ \tau ))
\to F^ M (\beta (u ) )~~~~~~\text{strongly in } 	L^q ( 0,T ; L^q   (\Omega )) ~~\forall q \geq 1 ,
\end{equation*}
and 
\begin{equation*}
\Pi _{\tau } f_ \tau  
\to f~~~~~~\text{strongly in } 	L^q ( 0,T ; L^q   (\Omega )) ~~\forall q \geq 1 .
\end{equation*}
Then \eqref{P001} weakly converges to 
$\partial _t  \xi + \eta = F ^M (\beta (u )) + f $. 

\vspace{3mm}

Let $t = T' \leq T$ be a time at which $\xi $ satisfy \eqref{Lem-AL}.  Namely, assume 
 \begin{equation*}
\int_{\Omega } j ^{\ast} ( \xi ( T' ) )dx - \int_{\Omega } j ^{\ast} ( \beta (u _0 ) ) dx
= \int_{0}^{T' } \LA \partial _t \xi  , u \RA _{W ^{1,p } } dt   . 
\end{equation*}
By applying Lemma \ref{A-L-Th} with 
$X _0 =  L^{\infty } (\Omega ) $ and 
$X = X_1 = W ^{-1, p' } (\Omega )$,
we can show that 
$ \{ \Lambda  _\tau \beta ( u _{\tau } ) \} _{\tau > 0 }$
strongly converges to $\xi $ in 
$C ([0,T] ; W ^{-1, p'} (\Omega ))$.
In particular, 
$\| \Lambda  _\tau \beta ( u _{\tau } ) (T' )  - \xi (T' )\| _{W ^{-1, p' } } \to 0 $ as $\tau \to 0 $. 
On the other hand, since  $\Lambda _\tau \beta (u _{\tau } ) (T' )$ is uniformly bounded in $ L ^{\infty } (\Omega )$, 
$\{ \Lambda _\tau \beta (u _{\tau } ) (T' ) \} _{\tau > 0 }$ has $\ast$-weakly convergent subsequence in $L ^{\infty } (\Omega )$. 
Let $\zeta $ be its limit. 
Then for every $v \in W ^{1, p } _0 (\Omega )$,
\begin{equation*}
\int_{\Omega }   \Lambda  _\tau \beta ( u _{\tau } )  (x, T' ) v (x) dx \to \int _{\Omega } \zeta  (x) v (x) dx 
=\LA  \zeta ,  v \RA _{W ^{1,p }}
\end{equation*}
and 
\begin{equation*}
\int_{\Omega }   \Lambda  _\tau \xi _{\tau }  (x, T'  ) v (x) dx 
=
\LA \Lambda  _\tau \xi _{\tau }  (T'  ) , v \RA _{W ^{1,p  }}
\to
\LA \xi  (T' ) , v \RA _{W ^{1,p  }}
\end{equation*}
hold as $\tau  \to 0 $,  hence $\zeta   = \xi ( T'  )$.

Let $ m  \tau  < T' \leq (m +1) \tau $ with some $m = 0 , \ldots, N -1 $. 
Since $j^{\ast }$ is convex, 
\begin{align*}
&\int_{\Omega } j^{\ast } (\Lambda _{\tau } \beta (u _{\tau })(T' )  ) dx  \\
=&
\int_{\Omega } j^{\ast } \LC 
\frac{\beta  ( u ^{n+1} _{\tau } )  -\beta (u ^{n} _{\tau } )  }{\tau } ( T'  - m\tau ) + \beta ( u ^{n} _{\tau } )
 \RC dx  \\
=&
\int_{\Omega } j^{\ast } \LC 
\frac{  T'  - m\tau   }{\tau }  \beta  ( u ^{n+1} _{\tau } )+
\frac{  ( m+1 ) \tau - T'     }{\tau }   \beta ( u ^{n} _{\tau } )
 \RC dx  \\
\leq &
\frac{  T' - m\tau   }{\tau } 
\int_{\Omega } j^{\ast } \LC 
 \beta  ( u ^{n+1} _{\tau } ) \RC dx  
+
\frac{  ( m+1 ) \tau - T'     }{\tau } 
\int_{\Omega } j^{\ast } \LC   \beta ( u ^{n} _{\tau } )
 \RC dx  \\
 \leq &
\int_{\Omega } j^{\ast } \LC 
 \beta  ( u ^{n+1} _{\tau } ) \RC dx  
+
\int_{\Omega } j^{\ast } \LC   \beta ( u ^{n} _{\tau } )
 \RC dx   
\end{align*}
By \eqref{P006}, the RHS is uniformly bounded and then 
$\{ \int_{\Omega } j^{\ast } (\Lambda _{\tau } \beta (u _{\tau }  ) (T' )) dx  \} _{\tau > 0 }$ is also uniformly bounded. 
Extract a convergent subsequence and let its limit be $\theta ' $. 
 Since  $w \mapsto \int_{\Omega } j^{\ast } (w ) dx $ is a lower semi-continuous convex function on 
 $ L ^{p'}(\Omega )$, we have 
 \begin{equation}
\int_{\Omega } j^{\ast } (\xi  (T' ) ) dx  
 \leq 
\liminf _{\tau \to 0 } 
\int_{\Omega } j^{\ast } (\Lambda _{\tau } \beta (u _{\tau }  ) ) (T' )dx 
\leq \theta ' . 
\label{P016}
\end{equation}

Multiplying \eqref{P001} by $u ^{n+1} _{\tau } $, we get  
\begin{align*}
& \int _{\Omega } j ^{\ast} (\beta (u ^{n+1} _ {\tau }) ) dx  - \int _{\Omega } j ^{\ast} (\beta (u ^{n} _ {\tau }) ) dx
+\tau  \LA - \nabla  \cdot \alpha (\cdot , \nabla   u ^{n+1 }_{\tau } ),  u ^{n+1} _{\tau } \RA _{W ^{1, p }} \\
&\leq 
\tau \int_{\Omega }  f ^n _{\tau }  u ^{n+1} _{\tau } dx 
+
\tau \int_{\Omega }  F ^M (\beta (u ^{n} _{\tau }))  u ^{n+1} _{\tau } dx . 
\end{align*}
Summing this from $n= 0 $ to $m -1 $, we have 
\begin{align*}
& \int _{\Omega } j ^{\ast} (\beta (u ^{m } _ {\tau }) ) dx  - \int _{\Omega } j ^{\ast} (\beta (u _0 ) ) dx
+\int_{0}^{m\tau }  \LA - \nabla  \cdot \alpha (\cdot , \nabla \Pi _\tau  u _{\tau } ) ,  u  _{\tau } \RA _{W ^{1, p }} dt \\
&\leq 
\int_{0}^{m\tau }  \int_{\Omega }  \Pi _{\tau } f _{\tau } \Pi _{\tau }   u  _{\tau } dx dt  
+
\int_{0}^{m\tau } \int_{\Omega } \Pi _{\tau }   F ^M (\beta (u  _{\tau }))  \Pi _{\tau } u _{\tau } dx dt . 
\end{align*}
Moreover, 
\begin{align*}
&\frac{T' - m \tau }{\tau} \int _{\Omega } j ^{\ast} (\beta (u ^{m+1} _ {\tau }) ) dx  
		-\frac{T' - m \tau }{\tau}   \int _{\Omega } j ^{\ast} (\beta (u ^{m} _ {\tau }) ) dx \\
&		
+(T' - m \tau )  \LA - \nabla  \cdot \alpha (\cdot , \nabla   u ^{m+1 }_{\tau } ),  u ^{m+1} _{\tau } \RA _{W ^{1, p }} \\
&~~~~~~ \leq 
(T' - m \tau ) \int_{\Omega }  f ^m _{\tau }  u ^{m+1} _{\tau } dx 
+
(T' - m \tau )  \tau \int_{\Omega }  F ^M (\beta (u ^{m} _{\tau }))  u ^{m+1} _{\tau } dx .
\end{align*}
Then we obtain 
\begin{align*}
&
 \frac{T' - m \tau }{\tau}  \int _{\Omega } j ^{\ast} (\beta (u ^{m+1 } _ {\tau }) ) dx 
		+\frac{ (m+1) \tau - T' }{\tau}   \int _{\Omega } j ^{\ast} (\beta (u  ^{m } _ {\tau }) ) dx \\
		&
				- \int _{\Omega } j ^{\ast} (\beta (u _0 ) ) dx
+\int_{0}^{T' }  \LA - \nabla  \cdot \alpha (\cdot , \nabla \Pi _\tau  u _{\tau } ) ,  u  _{\tau } \RA _{W ^{1, p }} dt \\
&~~~~~~ \leq 
\int_{0}^{T' }  \int_{\Omega }  \Pi _{\tau } f _{\tau } \Pi _{\tau }   u  _{\tau } dx dt  
+
\int_{0}^{T' } \int_{\Omega } \Pi _{\tau }   F ^M (\beta (u  _{\tau }))  \Pi _{\tau } u _{\tau } dx dt
\end{align*}
By the convexity of $j ^{\ast }$, 
\begin{align*}
&
 \frac{T' - m \tau }{\tau}  \int _{\Omega } j ^{\ast} (\beta (u ^{m+1 } _ {\tau }) ) dx 
		+\frac{ (m+1) \tau - T' }{\tau}   \int _{\Omega } j ^{\ast} (\beta (u  ^{m } _ {\tau }) ) dx \\
&\geq 
 \int _{\Omega } j ^{\ast} \LC   \frac{T' - m \tau }{\tau}  \beta (u ^{m+1 } _ {\tau })  +\frac{ (m+1) \tau - T' }{\tau}  \beta (u  ^{m } _ {\tau }) \RC dx \\ 
&= 
 \int _{\Omega } j ^{\ast} \LC   \frac{
\beta (u ^{m+1 } _ {\tau }) - \beta (u  ^{m } _ {\tau }) }{\tau}    (T' -m \tau )
			+  \beta (u  ^{m } _ {\tau }) \RC dx \\
&= 
 \int _{\Omega } j ^{\ast} \LC \Lambda _{\tau } \beta (u _ {\tau }) (T')  \RC dx 
\end{align*}
Therefore
\begin{align*}
&
 \int _{\Omega } j ^{\ast} \LC \Lambda _{\tau } \beta (u _ {\tau }) (T')  \RC dx 
				- \int _{\Omega } j ^{\ast} (\beta (u _0 ) ) dx
+\int_{0}^{T' }  \LA  - \nabla  \cdot \alpha (\cdot , \nabla \Pi _\tau  u _{\tau } ) ,  u  _{\tau } \RA _{W ^{1, p }} dt \\
&\leq 
\int_{0}^{T' }  \int_{\Omega }  \Pi _{\tau } f _{\tau } \Pi _{\tau }   u  _{\tau } dx dt  
+
\int_{0}^{T' } \int_{\Omega } \Pi _{\tau }   F ^M (\beta (u  _{\tau }))  \Pi _{\tau } u _{\tau } dx dt , 
\end{align*}
which yields as 
$\tau \to 0$
\begin{align*}
& \theta '  - \int _{\Omega } j ^{\ast} (\beta (u _0 ) ) dx
+\limsup _{\tau \to 0 }\int_{0}^{T'}  \LA  - \nabla  \cdot \alpha (\cdot , \nabla \Pi _\tau  u _{\tau } ) ,  u  _{\tau } \RA _{W ^{1, p }} dt \\
&\leq 
\int_{0}^{T'}  \int_{\Omega }   f u  dx dt  
+
\int_{0}^{T'} \int_{\Omega }   F ^M (\beta (u  ))   u dx dt . 
\end{align*}
By \eqref{P016} and Lemma \ref{IBP}, 
\begin{align*}
&\limsup _{\tau \to 0 }\int_{0}^{T ' }  \LA  - \nabla  \cdot \alpha (\cdot , \nabla \Pi _\tau  u _{\tau } ) ,  u  _{\tau } \RA _{W ^{1, p }} dt \\
\leq &
\int_{0}^{T'}  \int_{\Omega }   f u  dx dt  
+
\int_{0}^{T'} \int_{\Omega }   F ^M (\beta (u  ))   u dx dt \\
& ~~~~~~~~ -\int_{\Omega } j ^{\ast} (\xi (T')   ) dx
+ \int _{\Omega } j ^{\ast} (\beta (u _0 ) ) dx \\
\leq & 
\int_{0}^{T'}  \LA   F ^M (\beta (u  ))  + f  -\partial _t \beta (u ) , u  \RA _{W ^{1,p  }} dx dt  
=
\int_{0}^{T'}  \LA  \eta  , u  \RA _{W ^{1,p  }} dx dt  ,
\end{align*}
i.e., 
\begin{equation*}
\limsup _{\tau \to 0 }\int_{0}^{T}  \LA  - \nabla  \cdot \alpha (\cdot , \nabla \Pi _\tau  u _{\tau } )  - \eta ,  u  _{\tau } - u  \RA _{W ^{1, p }} dt \leq 0 .  
\end{equation*}
Hence by Lemma \ref{Demi}, we obtain $\eta = - \nabla  \cdot \alpha (\cdot , \nabla   u  )  $. 
Therefore $u $ is the solution to $\partial _t \beta (u ) - \nabla  \cdot \alpha (\cdot , \nabla   u  )  = F ^{M} (\beta (u) ) +f $,
and then 
we can assure that $F^  M ( \beta (u)) = F( \beta (u)) $ for a.e. $\Omega \times (0 , T '' ) $ with some sufficiently small $T'' \leq T $
by exactly the same argument as that for Theorem \ref{Th01}.


\subsection{Proof of Theorem  \ref{Th03}}

Define 
\begin{equation*}
\psi ^M (s ) 
:=
\int_{0}^{s} F^M ( \beta (\sigma )) d \sigma  . 
\end{equation*}
Since $F^ M $ is monotone increasing, 
$\psi ^M : \R \to \R $ is a convex function with domain $D (\psi ^M ) = \R $,
which satisfies 
\begin{equation*}
|\psi ^M (s )  | \leq 
M |s | . 
\end{equation*}
Moreover, 
\begin{equation*}
\psi ^M  (s) - \psi ^ M (s_ 0 ) = \int_{s_0}^{s }  F^M ( \beta (\sigma )) d \sigma
\geq F^M ( \beta (s _ 0  ))  (s  - s_ 0 )
\end{equation*}
holds for any $s \in \R  $, since  $F^ M \circ \beta  $ is monotone increasing. 
This implies that  $F^M ( \beta (s _ 0  )) \in \partial \psi ^M (s _ 0 )$.

Multiplying \eqref{P001} by $  u^{n+1} _{\tau } - u^{n} _{\tau } $, 
we obtain by the definition of the subdifferential 
\begin{equation*}
\int _{\Omega } F^M (\beta (u ^{n} _{\tau })) ( u^{n+1} _{\tau } - u^{n} _{\tau } ) dx
\leq
\psi ^M ( u^{n+1} _{\tau } ) - \psi ^M ( u^{n} _{\tau } ) 
\end{equation*}
and then by the monotonicity of $\beta $
\begin{equation*}
\int_{\Omega } a ( x , \nabla u ^{n+1 } _{\tau } ) dx 
-
\int_{\Omega } a ( x , \nabla u ^{n } _{\tau } ) dx 
\leq 
\psi ^M ( u^{n+1} _{\tau } ) - \psi ^M ( u^{n} _{\tau } ) .
\end{equation*}
Therefore we have 
\begin{align*}
\sup _{n =1, \ldots, N } \int_{\Omega } a ( x , \nabla u ^{n+1 } _{\tau } ) dx 
& \leq  \int_{\Omega } a ( x , \nabla u _0  ) dx 
+ \max _{n=1 \ldots, N } \int_{\Omega } \psi ^M (u ^{n}_\tau ) dx
-
 \int_{\Omega } \psi ^M (u_ 0  ) dx \\
& \leq  \int_{\Omega } a ( x , \nabla u _0  ) dx 
+
|\Omega | M 
(  \max _{n=1 \ldots, N } \| u ^{n}_\tau  \| _{L^{\infty } (\Omega )} 
+ \| u _ 0 \| _{L^{\infty } (\Omega )} ). 
\end{align*}
By \eqref{A01} and \eqref{P009}, the RHS is uniformly bounded as $\tau \to 0 $ if $u _ 0 \in W ^{1, p } _ 0 (\Omega ) \cap L ^{\infty } (\Omega )$.  
Hence in addition to \eqref{P005}, \eqref{P007}, and \eqref{P009},
we also obtain  
\begin{equation*}
\sup _{ 0\leq t \leq T } \| \nabla  \Pi _ \tau u _{ \tau }   \| ^p _{L^p }  
 \leq C 
\end{equation*}
and then the limit $u $ belongs to $L ^{\infty } (0,T ; W ^{1, p } _ 0 (\Omega ))$. 
The remaining arguments are exactly the same as in Theorem \ref{Th02}.


\section{Uniqueness of Solution}

In the end of this paper, we discuss the uniqueness and 
the comparison principle  of solutions.  
Assume is this section 
\begin{equation*}
\alpha (x , z ) = \alpha  (z) . 
\end{equation*}
Henceforth, let $ (s  ) _ + := \max \{ s ,  0 \} $ be the positive part of $s$. 
\begin{Th}
Let $ u _{ 0 i } , \beta (u _ {0i }) \in L^\infty $, $f _i \in L^ \infty  (0,T ;  L^ \infty  (\Omega ))$ ($i=1, 2$)
and $T' > 0 $ be a maximal existence time of solution to \eqref{P}. 
If $F $ is locally Lipschitz continuous,
\begin{equation}
\begin{split}
&\| \beta (u _1 (t ) ) -   \beta (u _2 ( t ) ) \| _{L^1 (\Omega ) } \\
& ~~~\leq e ^{Lt} \|  \beta (u _{01} ) -   \beta (u _{02} ) \|  _{L^1 (\Omega ) }
+ \int_{0}^{t} e^{L (t-s )}  \|   f _1 (x ,t ) -f _2 (x ,t ) \| _{L^1 (\Omega ) } dt
\end{split}
\label{Uni01} 
\end{equation}
holds for every  $t \in ( 0 , T ' )   $, where 
$u _ i $ is solution to \eqref{P} with $ u _ 0 = u_{0i }$ and $f = f _ i $, and  $L $ is the Lipschitz constant of $F$. 
Moreover, if $F $
is locally Lipschitz continuous and monotone increasing, 
\begin{equation}
\begin{split}
&\int _{\Omega } \LC \beta (u _1 (x, t ) ) -   \beta (u _2 (x, t ) )  \RC _+ dx \\
& ~~~\leq e ^{Lt} \int _{\Omega } \LC \beta (u _{01} (x ) ) -   \beta (u _{02} (x ) )  \RC _+ dx \\
&~~~~~~~~~~~+ \int_{0}^{t} e^{L (t-s )}  \int _{\Omega } \LC  f _1 (x,t) -f _2 (x,t)    \RC _+  dx dt
\end{split}
\label{Uni02}
\end{equation}
holds for every $t \in ( 0 , T ' )   $. 
Especially, if $ \beta (u _{01}  ) \leq    \beta (u _{02}  ) $, $f _1 \leq f _2 $ almost everywhere, 
then $ \beta (u _1 (x, t ) ) \leq   \beta (u _2 (x, t ) )  $ for a.e. $\Omega \times (0 , T ' )$. 
\end{Th}

\begin{proof}
By Theorem 5.6 of \cite{U}, we can see that 
\begin{align*}
&\int _{\Omega } \LC \beta (u _1 (x, t ) ) -   \beta (u _2 (x, t ) )  \RC _+ dx \\
& ~~~\leq \int _{\Omega } \LC \beta (u _{01} (x ) ) -   \beta (u _{02} (x ) )  \RC _+ dx \\
& ~~~+\int_{0}^{t}  \int _{\Omega } \LC F ( \beta (u _{1} (x, \tau  ) ) )  + f _1 (x,\tau ) -  F ( \beta (u _{2} (x, \tau   ) ) ) -f _2 (x,\tau )    \RC _+ dx d\tau  ,
\end{align*}
By $ (s _1  + s_2  ) _+ \leq (s _1  ) _+ +  ( s_2 ) _+  $, we have 
\begin{align*}
&\int _{\Omega } \LC \beta (u _1 (x, t ) ) -   \beta (u _2 (x, t ) )  \RC _+ dx \\
& ~~~\leq \int _{\Omega } \LC \beta (u _{01} (x ) ) -   \beta (u _{02} (x ) )  \RC _+ dx \\
& ~~~+\int_{0}^{t}  \int _{\Omega } \LC F ( \beta (u _{1} (x, \tau   ) ) ) -  F ( \beta (u _{2} (x, \tau   ) ) )  \RC _+ dx d \tau  
	+\int_{0}^{t}  \int _{\Omega } \LC  f _1 (x,\tau ) -f _2 (x,\tau )    \RC _+ dx d\tau  
\end{align*}
and by replacing $i =1$ with $i =2$, we get 
\begin{align*}
&\int _{\Omega } |  \beta (u _1 (x, t ) ) -   \beta (u _2 (x, t ) ) |  dx \\
& ~~~\leq \int _{\Omega } |  \beta (u _{01} (x ) ) -   \beta (u _{02} (x ) ) | dx \\
& ~~~+\int_{0}^{t}  \int _{\Omega } | F ( \beta (u _{1} (x, \tau   ) ) ) -  F ( \beta (u _{2} (x, \tau   ) ) )  | dx d\tau 
	+\int_{0}^{t}  \int _{\Omega } |  f _1 (x,\tau ) -   f _2 (x,\tau )  |  dx d \tau .
\end{align*}

Assume 
$F $ is locally Lipschitz continuous and monotone increasing, 
If $ \beta ( u _1  (x,t ) ) \geq  \beta ( u _2  (x,t ) ) $,
\begin{align*}
& F ( \beta ( u _1  (x,t ) )  ) -F ( \beta ( u _2  (x,t ) )  )  \\
=&
| F ( \beta ( u _1  (x,t ) )  ) -F ( \beta ( u _2  (x,t ) )  )  | \\
\leq &
L |\beta ( u _1  (x,t ) )   - \beta ( u _2  (x,t ) ) |
=
L ( \beta ( u _1  (x,t ) )   - \beta ( u _2  (x,t ) ) ) , 
\end{align*}
then 
\begin{equation*}
( F ( \beta ( u _1  (x,t ) )  ) -F ( \beta ( u _2  (x,t ) )  )  )_+ 
\leq 
L ( \beta ( u _1  (x,t ) )   - \beta ( u _2  (x,t ) ) ) _+ . 
\end{equation*}
Therefore 
\begin{align*}
&\int _{\Omega } \LC \beta (u _1 (x, t ) ) -   \beta (u _2 (x, t ) )  \RC _+ dx \\
& ~~~\leq \int _{\Omega } \LC \beta (u _{01} (x ) ) -   \beta (u _{02} (x ) )  \RC _+ dx \\
& ~~~+L \int_{0}^{t}  \int _{\Omega } \LC  \beta (u _{1} (x, \tau   ) )  -  \beta (u _{2} (x, \tau   ) )    \RC _+ dx d \tau 
	+\int_{0}^{t}  \int _{\Omega } \LC  f _1 (x,\tau  ) -f _2 (x, \tau  )    \RC _+ dx  d \tau .  
\end{align*}
Hence by the Gronwall inequality, we obtain \eqref{Uni02}. 
We can show \eqref{Uni01} by almost the same argument. 
\end{proof}

\subsection*{Acknowedgements}
The author is supported by 
JSPS Grant-in-Aid for Scientific Research (C) (No.24K06799)
and for Transformative Research Areas (B) (No.25H01453)
and Sumitomo Foundation Fiscal 2022 Grant for Basic Science Research Projects (No.2200250).

\bibliographystyle{unsrt}  


\end{document}